\documentclass{amsart}
\usepackage[hmargin=3cm,vmargin=3.5cm]{geometry}
\usepackage{amsmath,amsthm,amssymb,amscd}
\usepackage[usenames,dvipsnames]{color}
\definecolor{darkgreen}{rgb}{0,0.5,0}
\definecolor{darkblue}{rgb}{0,0,0.5}
\usepackage[colorlinks=true,urlcolor=darkblue,citecolor=darkgreen,linkcolor=darkblue]{hyperref}
\usepackage{graphicx}
\usepackage{esint}
\usepackage[all,cmtip]{xy}
\usepackage[latin1]{inputenc}
\usepackage[T1]{fontenc}
\usepackage{lmodern}
\usepackage{url}
\usepackage[draft]{fixme}
\newcommand{\relmiddle}[1]{\mathrel{}\middle#1\mathrel{}}

\renewcommand{\epsilon}{\varepsilon}
\newcommand{\eps}{\epsilon}
\renewcommand{\phi}{\varphi}

\newcommand{\abs}[1]{\lvert #1 \rvert}

\newcommand{\bra}[1]{\langle #1 |}

\newcommand{\bbC}{\mathbb{C}}

\newcommand{\bbZ}{\mathbb{Z}}

\newcommand{\bbR}{\mathbb{R}}

\newcommand{\Nabla}{\bnabla}

\newcommand{\Vol}{\mathrm{Vol}}

\newcommand{\calC}{\mathcal{C}}
\newcommand{\calD}{\mathcal{D}}

\newcommand{\calF}{\mathcal{F}}

\newcommand{\calL}{\mathcal{L}}
\newcommand{\calM}{\mathcal{M}}
\newcommand{\calN}{\mathcal{N}}

\newcommand{\calT}{\mathcal{T}}

\newcommand{\su}{\mathfrak{su}}
\newcommand{\SU}{\mathrm{SU}}

\newcommand{\SL}{\mathrm{SL}}

\newcommand{\End}{\text{End}}

\newcommand{\bnabla}{\boldsymbol{\nabla}}

\newcommand{\isom}{\cong}
\newcommand{\dbar}{\overline{\partial}}

\renewcommand{\Im}{\mathrm{Im}}

\DeclareMathOperator{\Gr}{Gr}
\DeclareMathOperator{\pdeg}{pdeg}

\DeclareMathOperator{\diag}{diag}

\DeclareMathOperator{\U}{U}

\DeclareMathOperator{\Id}{Id}

\DeclareMathOperator{\Tr}{Tr}
\DeclareMathOperator{\tr}{tr}

\DeclareMathOperator{\vol}{vol}

\DeclareMathOperator{\rk}{rk}

\DeclareMathOperator{\Ric}{Ric}

\def\XXint#1#2#3{{\setbox0=\hbox{$#1{#2#3}{\int}$}
\vcenter{\hbox{$#2#3$}}\kern-.5\wd0}}

\newcommand{\id}{\mathrm{id}}

\newcommand{\Aut}{\mathrm{Aut}}
\newcommand{\Hom}{\mathrm{Hom}}

\newcommand{\Ker}{\text{Ker}}

\newtheorem{thm}{Theorem}[section]
\newtheorem{theorem}[thm]{Theorem}

\newtheorem{prop}[thm]{Proposition}
\newtheorem{cor}[thm]{Corollary}

\newtheorem{lem}[thm]{Lemma}
\newtheorem{lemma}[thm]{Lemma}
\newtheorem{claim}[thm]{Claim}

\theoremstyle{remark}
\newtheorem{remark}[thm]{Remark}
\newtheorem{rem}[thm]{Remark}

\theoremstyle{definition}
\newtheorem{definition}[thm]{Definition}
\newtheorem{defn}[thm]{Definition}

\usepackage{amscd}
\usepackage{latexsym}
\usepackage{amsfonts}
\usepackage{ragged2e}

\usepackage{float}

\newcommand{\eproof}{\begin{flushright} $\square$ \end{flushright}}

\let\eproof\endproof %better construction, the \eproof construction
\newcommand{\cO}{{\mathcal O}}

\newcommand{\bC}{{\mathbb C}}
\newcommand{\bR}{{\mathbb R}}
\newcommand{\bT}{{\bar T}}

\newcommand{\Z}{{\mathbb Z}}
\newcommand{\bZ}{\Z{}}

\newcommand{\D}{{\mathcal D}}
\newcommand{\ra}{\mathop{\fam0 \rightarrow}\nolimits}

\newcommand{\cH}{ {\mathcal H}}

\newcommand{\T}{ {\mathcal T}}

\newcommand{\V}{ {\mathcal V}}\newcommand{\N}{ {\mathbb N}}
\renewcommand{\L}{{\mathcal L}}

\renewcommand{\P}{ {\mathbb P}}

\newcommand{\tG}{ {\tilde G}}
\newcommand{\tW}{ {\tilde W}}
\newcommand{\s}{\sigma}

\newcommand{\Nablat}{{\bnabla}^{t}}

\providecommand{\bracket}[1]{\left \langle #1 \right \rangle}
\newcommand{\bP}{\mathbb{P}}

\newcommand{\lie}{\mathfrak}

\begin{document}
\title[Unitarity of the Hitchin connection]{Mapping  class group invariant unitarity of the Hitchin connection over Teichmüller space}
\author{Jørgen Ellegaard Andersen}
\address{Center for Quantum Geometry of Moduli
  Spaces\\ 
  University of Aarhus\\
  DK-8000, Denmark}

\email{andersen@qgm.au.dk}
\thanks{Supported in part by the center of excellence grant "Center for Quantum Geometry of Moduli Space" from the Danish National Research Foundation.}

\maketitle
\begin{abstract}
  We provide a geometric construction of the unitary structure which is projectively preserved by the Hitchin connection. We analyze the asymptotic behavior of it and we establish that it is uniformly in the level equivalent to the Hermitian structure induced by the $L_2$ inner product on smooth sections.
\end{abstract}
\section{Introduction}
Let $\Sigma$ be a closed surface of genus $g > 1$ and choose a point $p$ on $\Sigma$. Let $\Gamma$ be the mapping class group of $\Sigma$. We will denote the moduli space of flat $\SU(2)$-connections on $\Sigma - \{p\}$ with holonomy $-\Id \in \SU(2)$ around $p$ by $M$. It is well known that $M$ carries the Goldman symplectic structure $\omega$, which is determined by choosing an invariant inner product on the Lie algebra of $\SU(2)$. For the appropriate choice of scaling of this inner product we get that the class of $\omega$ generates $H^2(M,\bbZ)$. Let now $(\calL,\nabla,\langle\cdot,\cdot\rangle)$ be a prequantum line bundle over $(M,\omega)$, i.e. the curvature of $\nabla$ is the symplectic form
\begin{align*}
		F_\nabla = -i\omega.
\end{align*}
It is well known that $\Gamma$  acts by symplectomorphisms on $(M,\omega)$ and that this action can be lifted to an action of $\Gamma$ on $\L$ which preserves $\nabla$ and $\langle \cdot,\cdot \rangle$ (see e.g. \cite{Fr} and \cite{A1}). There is a very natural $\Gamma$-equivariant family of complex structures parametrized by Teichmüller space $\calT$ on $M$. Suppose $\sigma \in \calT$ is a complex structure on $\Sigma$. Then we can consider the moduli space of stable holomorphic bundles of rank $2$ and determinant isomorphic to $[p]$ on the Riemann surface $\Sigma_\sigma$. This moduli space is naturally a complex manifold $M_\sigma$ and by the theorem of Narasimhan and Seshadri, we get a natural diffeomorphism of the underlying smooth manifold of $M_\sigma$ to $M$. This structure in fact depends holomorphically on $\sigma \in \calT$. The complex structure on $M_\sigma$ combines with the connection $\nabla$ to produce the structure of a holomorphic line bundle on $\L$ over $M_\sigma$, and one gets a vector bundle $H^{(k)}$ over $\calT$, with fiber at $\sigma$ given by
\begin{align*}
	H_\sigma^{(k)} = H^0(M_\sigma,\calL^k).
\end{align*}
There is a natural holomorphic structure in the bundle $H^{(k)}$ over $\calT$. The main result pertaining to this bundle is:
\begin{thm}[Axelrod, Della Pietra and Witten; Hitchin]\label{MainGHCI}
  The bundle $H^{(k)}$ suppors a natural projectively flat $\Gamma$-invariant connection $\bnabla$.
\end{thm}
This is a result proved independently by Axelrod, Della Pietra and Witten \cite{ADW} and by Hitchin \cite{H}. In Section~\ref{sect2}, we review our differential geometric construction of the connection $\bnabla$ in the general setting discussed in \cite{A9}.
\begin{defn}
  For a given $k \in \N$, a Hermitian structure $(\cdot,\cdot)$ on $H^{(k)}$ is said to be \emph{projectively preserved} by the Hitchin connection, if there exists a $1$-form $\alpha \in \Omega^1(\calT)$ such that for all vector fields $V$ on $\calT$ and all $s_1,s_2 \in C^\infty(\calT,H^{(k)})$, we have that
  \begin{align*}
		V[(s_1,s_2)] - (\bnabla_V s_1,s_2) - (s_1,\bnabla_V s_2) = \alpha(V)(s_1,s_2).
  \end{align*}
  Two Hermitian structures on $H^{(k)}$ are called \emph{projectively equivalent} if there exists a smooth function $c$ defined on $\calT$ such that $c \Id$ induces an isometry between the two structures. A \emph{projective Hermitian structure} is by definition an equivalence class of a Hermitian structure on $H^{(k)}$.
\end{defn}
\begin{rem}
  We observe that a projective Hermitian structure $(\cdot,\cdot)$ on $H^{(k)}$ induces a Hermitian structure on $\End(H^{(k)})$ and that two Hermitian structures on $H^{(k)}$ induces the same Hermition structure on $\End(H^{(k)})$ if and only if they are equivalent. Moreover, a Hermitian structure on $H^{(k)}$ is projectively preserved by the Hitchin connection if and only if the induced Hermitian structure on $\End(H^{(k)})$ is preserved by the connection in $\End(H^{(k)})$ induced by the Hitchin connection.
\end{rem}
We construct in this paper, for each $k$, a specific projective Hermitian structure $(\cdot,\cdot)^{(k)}$ in $H^{(k)}$, which is projectively preserved by the Hitchin connection $\bnabla$. This structure is determined by having a certain asymptotics at particular boundary points of Teichmüller space, which corresponds to pair of pants decompositions of the surface $\Sigma$ . Let us now discuss this asymptotics. First we review the constructions of \cite{JW} and the degeneration result of \cite{A3}.

Suppose $P$ is a pair of pants decomposition of $\Sigma$. By mapping a flat $\SU(2)$-connection to the traces of its holonomy around each of the curves in $P$, we get a smooth map $h_P : M \to [-2,2]^{3g-3}$. The fibers of this map are the so-called Jeffrey--Weitsman real polarization $F_P$ on the moduli space $M$. The fibers over the part of the image which is contained in $(-2,2)^{3g-3}$ are Lagrangian sub-tori of $M$. Fibers which map to the boundary of the image $h_P(M) \subset [-2,2]^{3g-3}$ are singular. We will give a precise description of them in Section~\ref{sect3}.

The geometric quantization of the moduli space $M$ with respect to the real polarizations $F_P$ was studied by Jeffrey and Weitsman in \cite{JW}. In general, when one quantizes a compact symplectic manifold with respect to a real polarization with compact leaves, one needs to consider distributional sections of the prequantum line bundle, which are covariant constant along the polarization (see e.g. \cite{Woodhouse}, \cite{A3} and \cite{A1}). One finds that these distributional sections are supported on the so-called Bohr--Sommerfeld fibers of the polarizations.
\begin{defn}
  Let $H_P^{(k)}$ denote the vector space of distributional sections of $\calL^k$ over $M$, which are covariant constant along the directions of $F_P$. A leaf $L$ of $F_P$, i.e. a fiber of $h_P$, is called a level $k$ Bohr--Sommerfeld fiber if $(\calL^k,\nabla)|_L$ is trivial. We denote the set of level $k$ Bohr--Sommerfeld fibers by $B_k(P)$.
\end{defn}
We observe that if $L$ is a leaf of $F_P$, then $L$ is a level $k$ Bohr--Sommerfeld fiber if and only if $(\calL^k,\nabla)|_L$ admits a covariant constant section defined on all of $L$. By choosing a covariant constant section of $(\calL^k,\nabla)|_L$ for each $L \in B_k(P)$ and considering them as distributional section of $\calL^k$ over $M$, we obtain a basis for $H_P^{(k)}$. The main result of \cite{JW} is that
\begin{align*}
	\dim H_\sigma^{(k)} = \dim H_P^{(k)}
\end{align*}
for all $\sigma \in \calT$ and every pair of pants decomposition of $\Sigma$.

Consider a family $\sigma_t$, $t \in \bbR_+ \cup \{0\}$ obtained from some arbitrary starting point $\sigma_0 \in \calT$, such that $\sigma_t$ is obtained from $\sigma_0$ by insertion of a flat cylinder of length $t$ into the cut of $\Sigma$ along each of the curves in $P$. We have the following theorem from \cite{A2}.
\begin{thm}
  The complex polarizations on $M$ induced from $\sigma_t$ converge to $F_P$ as $t$ goes to infinity.
\end{thm}
Let
\begin{align*}
	P_t(\sigma_0,P) : H_{\sigma_0}^{(k)} \to H_{\sigma_t}^{(k)}
\end{align*}
be the parallel transport with respect to the Hitchin connection in $H^{(k)}$ over $\calT$ along the curve $(\sigma_s)$, $s \in [0,t]$. In Section~\ref{sect3}, we will show that there exists a limiting linear map
\begin{align}
  \label{eq1}
	P_\infty(\sigma_0,P) : H_{\sigma_0}^{(k)} \to H_P^{(k)}.
\end{align}
We further establish the following result in Section~\ref{sect4}.
\begin{thm}
  \label{thm3}
  The map \eqref{eq1} is an isomorphism.
\end{thm}
Jeffrey and Weitsman also describe the set $B_k(P)$ explicitly in \cite{JW} as follows. We can associate to $P$ a trivalent graph $\Sigma_P$ as follows. Each pair of pants is represented by a vertex and two vertices are connected by an edge if they are adjacent on the surface $\Sigma$.

By the definition of $h_P$ above, we see that the set of leaves of $F_P$ is identified with a subset of the set of maps from the set of edges $E_{\Gamma_P}$ of $\Gamma_P$ to $[-2,2]$. By identifying $[0,k]$ with $[-2,2]$ using the bijection
\begin{align*}
	t \mapsto 2\cos(\pi t/k),
\end{align*}
we can consider the set of leaves of $F_P$ as a subset of the set of maps from $E_{\Gamma_P}$ to $[0,k]$. For each vertex $v$ in the set of vertices $V_{\Gamma_P}$ in $\Gamma_P$, we let $e_1(v),e_2(v)$ and $e_3(v)$ be the three edges emanating from $v$.

\begin{defn}
  For each pair of pants decomposition $P$ of $\Sigma$,
  \begin{align*}
		L_k(P) = \left\{ l : E_{\Gamma_P} \to \{0,\dots,k\} \relmiddle| \begin{aligned} &\text{$l(e) \in 2\bbZ$ if $e \in E_{\Gamma_P}$ is separating} \\ &\text{$(l(e_1(v)),l(e_2(v)),l(e_3(v)))$ is admissible $\forall v \in V_{\Gamma_P}$} \end{aligned} \right\},
  \end{align*}
  where a triple of integers $(l_1,l_2,l_3)$ is said to be \emph{admissible} if the following three conditions are satisfied.
  \begin{gather*}
		\abs{l_1 - l_2} \leq l_3 \leq l_1 + l_2, \\
		l_1+l_2+l_3 \leq 2k, \\
		l_1+l_2+l_3 \in 2\bbZ.
  \end{gather*}
\end{defn}
Theorem 8.1 in \cite{JW} states that
\begin{thm}[Jeffrey--Weitsman]
  Under the above identification we have that
  \begin{align*}
		B_k(P) = L_k(P).
  \end{align*}
\end{thm}
We recall that the Reshetikhin--Turaev TQFT assigns a Hermitian vector space to $\Sigma$, which given the pair of pants decomposition $P$ of $\Sigma$ is provided with a basis indexed exactly by $L_k(P)$ \cite{RT1} , \cite{RT2}, \cite{T}. We also refer to the skein theory model of Blanchet, Habegger, Masbaum and Vogel, \cite{BHMV1}, \cite{BHMV2}, \cite{B1}. We let the vector corresponding to $l \in L_k(P)$ be denoted by $v_l$. By Theorem  4.11 in \cite{BHMV1} we have that the basis is orthogonal and the norms are given by the following formula
\begin{align}
  \label{eq2}
	[v_l,v_l] = \eta^{1-g} \frac{\prod_{v \in V_{\Gamma_P}} \langle l(v) \rangle}{\prod_{e \in E_{\Gamma_P}} \langle l(e)\rangle},
\end{align}
where
\begin{align*}
	\eta = \sqrt{\frac{2}{r}} \sin(\pi /r)
\end{align*}
with $\langle j \rangle = (-1)^j [j+1]$ for any integer $j$, and for any triple of integers $(a,b,c)$,
\begin{align*}
	\langle a,b,c \rangle = (-1)^{\alpha+\beta+\gamma}\frac{[\alpha+\beta+\gamma+1]![\alpha]![\beta]![\gamma]!}{[a]![b]![c]!}
\end{align*}
with
\begin{align*}
	a = \beta+\gamma, \, b = \alpha + \gamma, \, c = \alpha + \beta, 
\end{align*}
and
\begin{align*}
[j] = \frac{\sin(j \pi/r)}{\sin(\pi/r)}.
\end{align*}
Furthermore, $r = k + 2$. We observe that (\ref{eq2}) is positive for all $l\in L_k(P)$. We now introduce an orthonormal basis $\tilde v_l$, $l\in L_k(P)$, given by

\begin{align*}
	\tilde v_l = \frac{v_l}{[v_l,v_l]^\frac12}.
\end{align*}
As will be demonstrated in this paper, the basis vector $\tilde v_l$ correspond to a covariant constant section of $\L^k$ of unit norm over the leaf of $F_P$ corresponding to $l$. We therefore define a Hermitian structure $(\cdot,\cdot)_P^{(k)}$ in $H_P^{(k)}$ as follows. Suppose $s_1,s_2 \in H_P^{(k)}$, then for each $L \in B_k(P)$ we have that $s_i$, $i = 1,2$, are covariant constant sections of $\calL^k|_L$. Hence we see that $\langle s_1, s_2 \rangle$ is constant along the leaves of $P$ in $B_k(P)$ and thus $\langle s_1, s_2 \rangle$ becomes a function on $B_k(P)$. Under the above identification of $B_k(P)$ with $L_k(P)$, we can thus interpret $\langle s_1, s_2 \rangle$ as a function defined on $L_k(P)$.
\begin{defn}
  For any $s_1,s_2 \in H_P^{(k)}$ we define
  \begin{align*}
		(s_1,s_2)_P^{(k)} = \sum_{l \in L_k(P)} \langle s_1, s_2 \rangle(l) .
  \end{align*}
\end{defn}
We observe that $(\cdot,\cdot)^{(k)}_P$ is positive definite. We proof the following theorem in Section \ref{sect7}.
\begin{thm}
\label{HS}
There is a unique projective Hermitian structure $(\cdot,\cdot)^{(k)}$ in $H^{(k)}$ which is projectively preserved by the Hitchin connection, projectively invariant under the mapping class group action and satisfies the following asymptotics: For any $\sigma_0 \in \calT$ and pair of pants decomposition $P$ of $\Sigma$, we have that
\begin{align*}
	P_\infty(\sigma_0,P) : (H_{\sigma_0}^{(k)},(\cdot,\cdot)_{\sigma_0}^{(k)}) \to (H_P^{(k)},(\cdot,\cdot)_P^{(k)})
\end{align*}
is a projective isometry.
\end{thm}

We recall that we also have the $L^2$-Hermitian structure on $H^{(k)}$ given by
\begin{align*}
	(s_1,s_2)_{L^2}^{(k)} = \int_M \langle s_1, s_2 \rangle\frac{\omega^n}{n!},
\end{align*}
for any two sections $s_1$ and $s_2$ of $H^{(k)}$. We now wish to compare $(\cdot,\cdot)_{L^2}^{(k)}$ with $(\cdot,\cdot)^{(k)}$. Let $\{\cdot , \cdot\}_{L^2}^{(k)}$ respectively $\{\cdot,\cdot\}^{(k)}$ be the Hermitian structures induced on $\End(H^{(k)}) \isom H^{(k)} \otimes (H^{(k)})^*$ by $(\cdot,\cdot)^{(k)}_{L^2}$ respectively $(\cdot,\cdot)^{(k)}$.

We will show in Theorem~\ref{thm18} that the inner product $(s_1,s_2)^{(k)}$ has a representative of the following form:
\begin{thm}
  \label{thm6}
  There exist functions $G^{(k)} \in C^\infty(\calT,C^\infty(M))$, such that
  \begin{align*}
		(s_1,s_2)_\sigma^{(k)} = \int_M \langle s_1, s_2\rangle G_\sigma^{(k)} \frac{\omega^m}{m!}
  \end{align*}
  for $s_1 , s_2 \in H^0(M_\sigma, \calL^k)$, which has the asymptotic expansion
  \begin{align*}
		G_\sigma^{(k)} = \exp(-F_\sigma + O(1/k))
  \end{align*}
  for all $\sigma \in \calT$, where $F_\sigma \in C^\infty(M)$ is the Ricci potential for $(M_\sigma,\omega)$.
\end{thm}
From this theorem, we immediately get the following corollary.
\begin{cor}
  The Hermitian structures $\{\cdot,\cdot\}^{(k)}$ on $\End(H^{(k)})$ are uniformly equivalent to $\{\cdot,\cdot\}^{(k)}_{L^2}$, i.e. for each $\sigma \in \calT$ there is a constant $c_\sigma$ independent of $k$ such that
  \begin{align*}
		c_\sigma^{-1}\abs{A}_{L^2}^{(k)} \leq \abs{A}_\sigma^{(k)} \leq c_\sigma \abs{A}_{L^2}^{(k)}
  \end{align*}
  for all $A \in \End(H_\sigma^{(k)})$ and all $k$, where $\abs{\cdot}_{L^2}^{(k)}$ respectively $\abs{\cdot}^{(k)}$ are the norms associated to $\{\cdot,\cdot\}_{L^2}^{(k)}$ respectively to $\{\cdot,\cdot\}^{(k)}$.
\end{cor}
\textbf{Acknowledgements.} We thank Gregor Massbaum, Nicolai Reshetikhin, Bob Penner, Søren Fuglede Jørgensen, Jakob Lindblad Blaavand, Jens-Jakob Kratmann Nissen and Jens Kristian Egsgaard for helpful discussion.

\section{The Hitchin connection}\label{ghc}\label{sect2}

In this section, we review our construction of the Hitchin connection
using the global differential geometric setting of \cite{A9}. This
approach is close in spirit to Axelrod, Della Pietra and Witten's in
\cite{ADW}, however we do not use any infinite dimensional gauge
theory. In fact, the setting is more general than the gauge theory
setting in which Hitchin in \cite{H} constructed his original
connection. But when applied to the gauge theory situation, we get the
corollary that Hitchin's connection agrees with Axelrod, Della Pietra
and Witten's.

Hence, we start in the general setting and let $(M,\omega)$ be any
compact symplectic manifold.

\begin{definition}\label{prequantumb}
  A prequantum line bundle $(\L, (\cdot,\cdot), \nabla)$ over the
  symplectic manifold $(M,\omega)$ consist of a complex line bundle
  $\L$ with a Hermitian structure $(\cdot,\cdot)$ and a compatible
  connection $\nabla$ whose curvature is
  \begin{align*}
    F_\nabla(X,Y) = [\nabla_X, \nabla_Y] - \nabla_{[X,Y]} = -i \omega
    (X,Y).
  \end{align*}
  We say that the symplectic manifold $(M,\omega)$ is prequantizable
  if there exist a prequantum line bundle over it.
\end{definition}

Recall that the condition for the existence of a prequantum line
bundle is that 
$$\biggl[\frac{\omega}{2\pi}\biggr]\in \Im(H^2(M,\bZ) \ra
H^2(M,\bR)).$$
 Furthermore, the inequivalent choices of prequantum line
bundles (if they exist) are parametriced by $H^1(M,U(1))$ (see
e.g. \cite{Woodhouse}).

We shall assume that $(M,\omega)$ is prequantizable and fix a
prequantum line bundle $(\L, (\cdot,\cdot), \nabla)$.

Assume that $\T$ is a smooth manifold which smoothly parametrizes
K\"{a}hler structures on $(M,\omega)$. This means that we have a
smooth\footnote{Here a smooth map from $\T$ to $C^\infty(M,W)$, for
  any smooth vector bundle $W$ over $M$, means a smooth section of
  $\pi_M^*(W)$ over $\T\times M$, where $\pi_M$ is the projection onto
  $M$. Likewise, a smooth $p$-form on $\T$ with values in
  $C^\infty(M,W)$ is, by definition, a smooth section of
  $\pi_{\T}^*\Lambda^p(\T)\otimes \pi_M^*(W)$ over $\T\times M$. We
  will also encounter the situation where we have a bundle $\tW$ over
  $\T\times M$ and then we will talk about a smooth $p$-form on $\T$
  with values in $C^\infty(M,\tW_\s)$ and mean a smooth section of
  $\pi_{\T}^*\Lambda^p(\T)\otimes \tW$ over $\T\times M$.}  map $I :
\T \ra C^\infty(M,\End(TM))$ such that $(M,\omega, I_\s)$ is a
K\"{a}hler manifold for each $\s\in \T$.

We will use the notation $M_\sigma$ for the complex manifold $(M,
I_\s)$. For each $\s\in \T$, we use $I_\s$ to split the complexified
tangent bundle $TM_\bC$ into the holomorphic and the anti-holomorphic
parts. These we denote by
$$T_{\s} = E(I_\s,i) = \Im(\Id - iI_\s)$$
and
$$\bT_{\s}= E(I_\s,-i) = \Im(\Id + iI_\s)$$
respectively.

The real K\"{a}hler-metric $g_\s$ on $(M_\s,\omega)$, extended complex
linearly to $TM_\bC$, is by definition
\begin{align}
  \label{eq:3}
  g_\s(X,Y) = \omega(X,I_\s Y),
\end{align}
where $X,Y \in C^\infty(M,TM_\bC)$.

The divergence of a vector field $X$ is the unique function
$\delta(X)$ determined by
\begin{align}
  \label{eq:1}
  \mathcal{L}_X \omega^m = \delta(X) \omega^m,
\end{align}
with $m = \dim M$.
It can be calculated by the formula $\delta(X) = \Lambda d (i_X
\omega)$, where $\Lambda$ denotes contraction with the K\"ahler form.
Even though the divergence only depend on the volume, which is
independent of the of the particular K\"ahler structure, it can be
expressed in terms of the Levi-Civita connection on $M_\sigma$ by
$\delta(X) = \tr \nabla_\sigma X$.

Inspired by this expression, we define the divergence of a symmetric
bivector field $$B \in C^\infty(M, S^2(TM_{\bC}))$$ by
\begin{align*}
  \delta_\sigma(B) = \tr \nabla_\sigma B.
\end{align*}
Notice that the divergence of bivector fields does depend on the point
$\sigma \in \mathcal{T}$.

Suppose $V$ is a vector field on $\T$. Then we can differentiate $I$
along $V$ and we denote this derivative by $V[I] : \T \ra
C^\infty(M,\End(TM_\bC))$. Differentiating the equation $I^2 = -\Id$,
we see that $V[I]$ anti-commutes with $I$. Hence, we get that
\[V[I]_\s \in C^\infty(M, (\bT_\s^*\otimes T_\s)\oplus
(T_\s^*\otimes \bT_\s))\] for each $\s\in \T$. Let
\begin{align*}
  V[I]_\s = V[I]'_\s + V[I]''_\s
\end{align*}
be the corresponding decomposition such that $V[I]'_\s\in C^\infty(M,
\bT_\s^*\otimes T_\s)$ and $V[I]''_\s\in C^\infty(M, T_\s^*\otimes
\bT_\s)$.

Now we will further assume that $\T$ is a complex manifold and that
$I$ is a holomorphic map from $\T$ to the space of all complex
structures on $M$.  Concretely, this means that
\[V'[I]_\s = V[I]'_\s\] and
\[V''[I]_\s = V[I]''_\s\] for all $\s\in \T$, where $V'$ means the
$(1,0)$-part of $V$ and $V''$ means the $(0,1)$-part of $V$ over $\T$.

Let us define $\tG(V) \in C^\infty(M , TM_\bC\otimes TM_\bC)$ by
\[V[I] = \tG(V) \omega,\] and define $G(V) \in C^\infty(M, T_\s
\otimes T_\s)$ such that
\[\tG(V) = G(V) + {\overline G(V)} \]
for all real vector fields $V$ on $\T$. 

We see that $\tG$ and $G$ are
one-forms on $\T$ with values in $C^\infty(M , TM_\bC\otimes TM_\bC)$
and $C^\infty(M, T_\s \otimes T_\s)$, respectively.  We observe that
\[V'[I] = G(V)\omega,\] and $G(V) = G(V')$.

Using the relation \eqref{eq:3}, one checks that
\begin{align*}
  \tilde G(V) = - V[g^{-1}],
\end{align*}
where $g^{-1} \in C^\infty(M, S^2(TM))$ is the symmetric bivector
field obtained by raising both indices on the metric tensor.  Clearly,
this implies that $\tG$ takes values in $C^\infty(M , S^2(TM_\bC))$
and thus $G$ takes values in $C^\infty(M, S^2(T_\s))$.

On $\L^k$, we have the smooth family of $\bar\partial$-operators
$\nabla^{0,1}$ defined at $\s\in \T$ by
\[\nabla^{0,1}_\s = \frac12 (1+i I_\s)\nabla.\]
For every $\sigma\in \T$, we consider the finite-dimensional subspace
of $C^\infty(M,\L^k)$ given by
\[H_\sigma^{(k)} = H^0(M_\s, \L^k) = \{s\in C^\infty(M, \L^k)\,| \, 
\nabla^{0,1}_\s s =0 \}.\] Let $\Nablat$ denote the trivial connection
in the trivial bundle $\mathcal{H}^{(k)} = \T\times C^\infty(M,\L^k)$,
and let $\D(M,\L^k)$ denote the vector space of differential operators
on $C^\infty(M,\L^k)$. For any smooth one-form $u$ on $\T$ with values
in $\D(M,\L^k)$, we have a connection $\Nabla$ in $\cH^{(k)}$ given by
$$\Nabla_V = \Nablat_V - u(V)$$
for any vector field $V$ on $\T$.

\begin{lemma}
  The connection $\Nabla$ in $\cH^{(k)}$ preserves the subspaces
  $H^{(k)}_\sigma \subset C^\infty(M,\L^k)$, for all $\sigma \in \T$,
  if and only if
  \begin{equation}
    \frac{i}2  V[I] \nabla^{1,0} s + \nabla^{0,1}u(V)s = 0\label{eqcond}
  \end{equation}
  for all vector fields $V$ on $\T$ and all smooth sections $s$ of
  $H^{(k)}$.
\end{lemma}

This result is not surprising. See \cite{A9} for a proof this
lemma. Observe that if this condition holds, we can conclude that the
collection of subspaces $H^{(k)}_\sigma \subset C^\infty(M,\L^k)$, for
all $\sigma \in \T$, form a subbundle $H^{(k)}$ of $\cH^{(k)}$.

We observe that $u(V'') = 0$ solves \eqref{eqcond} along the
anti-holomorphic directions on $\T$ since
\[V''[I] \nabla^{1,0} s = 0.\] In other words, the $(0,1)$-part of the
trivial connection $\Nablat$ induces a $\bar\partial$-operator on
$H^{(k)}$ and hence makes it a holomorphic vector bundle over $\T$.

This is of course not in general the situation in the $(1,0)$-direction. Let us now consider a particular $u$ and prove that it
solves \eqref{eqcond} under certain conditions.

On the K\"{a}hler manifold $(M_\s,\omega)$, we have the K\"{a}hler
metric and we have the Levi-Civita connection $\nabla$ in $T_\s$. We
also have the Ricci potential $F_\s\in C^\infty_0(M,\bR)$. here
\[C^\infty_0(M,\bR) = \left\{ f\in C^\infty(M,\bR) \mid \int_M f
  \omega^m = 0\right\}. \] The Ricci potential is the element of
$F_\s\in C^\infty_0(M,\bR)$ which satisfies
\[\Ric_\s = \Ric_\s^H + 2 i \partial_\s\dbar_\s F_\s,\]
where $\Ric_\s\in \Omega^{1,1}(M_\s)$ is the Ricci form and
$\Ric_\s^H$ is its harmonic part. In this way we get a
smooth function $F : \T \ra C^\infty_0(M,\bR)$.

For any symmetric bivector field $B\in C^\infty(M, S^2(TM))$ we get a
linear bundle map
\begin{align*}
  B \colon TM^* \ra TM
\end{align*}
given by contraction. In particular, for a smooth function $f$ on $M$,
we get a vector field 
$$B d f \in C^\infty(M,TM).$$

We define the operator
\begin{eqnarray*}
  \Delta_B &: &C^\infty(M,\L^k) \xrightarrow{\nabla} C^\infty(M,TM^*\otimes\L^k)
  \xrightarrow{B\otimes\Id}
  C^\infty(M,TM \otimes \L^k) \\
  && \qquad \xrightarrow{\nabla_\s\otimes \Id +
    \Id\otimes \nabla}
  C^\infty(M,TM^* \otimes TM \otimes\L^k)
  \xrightarrow{\tr} C^\infty(M,\L^k).
\end{eqnarray*}
Let's give a more concise formula for this operator.  Define the
operator
\begin{align*}
  \nabla^2_{X,Y} = \nabla_X \nabla_Y - \nabla_{\nabla_X Y},
\end{align*}
which is tensorial and symmetric in the vector fields $X$ and
$Y$. Thus, it can be evaluated on a symmetric bivector field and we
have
\begin{align*}
  \Delta_B = \nabla^2_B + \nabla_{\delta(B)}.
\end{align*}

Putting these constructions together, we consider, for some $n\in \bZ$
such that $2k+n \neq 0$, the following operator
\begin{equation}
  u(V) = \frac1{k+n/2}o(V) - V'[F],\label{equ}
\end{equation}
where
\begin{equation}
  o(V) = - \frac{1}{4} (\Delta_{G(V)} + 2\nabla_{G(V)dF} - 2n V'[F]).\label{eqo}
\end{equation}

The connection associated to this $u$ is denoted $\Nabla$, and we call
it the {\em Hitchin connection} in $\cH^{(k)}$. Following \cite{A9}, we now introduce the notion of a rigid family of K\"{a}hler structures.

\begin{definition}\label{Ridig}
  We say that the complex family $I$ of K\"{a}hler structures on
  $(M,\omega)$ is {\em rigid} if
  \[\dbar_\sigma (G(V)_\sigma) = 0 \]
  for all vector fields $V$ on $\T$ and all points $\sigma\in \T$.
\end{definition}

We will assume our holomorphic family $I$ is rigid. There are plenty
of examples of rigid holomorphic families of complex structures, see
e.g. \cite{AGL}.

\begin{theorem}\label{HCE}
  Suppose that $I$ is a rigid family of K\"{a}hler structures on the
  compact, prequantizable symplectic manifold $(M,\omega)$ which
  satisfies that there exists an $n\in \bZ$ such that the first Chern
  class of $(M,\omega)$ is $n [\frac{\omega}{2\pi}]\in H^2(M,\bZ)$ and
  $H^1(M,\bR) = 0$. Then $u$ given by \eqref{equ} and \eqref{eqo}
  satisfies \eqref{eqcond} for all $k$ such that $2k+n \neq 0$.
\end{theorem}

Hence, the Hitchin connection $\Nabla$ preserves the subbundle
$H^{(k)}$ under the stated conditions. Theorem \ref{HCE} is
established in \cite{A9} through the following three lemmas.

\begin{lemma} \label{dbarl} Assume that the first Chern class of
  $(M,\omega)$ is $n [\frac{\omega}{2\pi}]\in H^2(M,\bZ)$. For any
  $\s\in \T$ and for any $G\in H^0(M_\s, S^2(T_\s))$, we have the
  following formula
  \begin{align*}
    \nabla^{0,1}_\s (\Delta_G(s) + 2 \nabla_{G d F_\s}(s)) = - i (2 k
    + n) \omega G \nabla (s) + 2 ik \omega (G dF_\sigma)s + ik \omega
    \delta_\sigma(G) s,
  \end{align*}
  for all $s\in H^0(M_\s, \L^k)$.
\end{lemma}

\begin{lemma}\label{Vriccipot}
  We have the following relation
  \begin{align*}
    4i \bar\partial_\s (V'[F]_\s) = 2 (G(V) dF)_\sigma \omega +
    \delta_\sigma (G(V))_\sigma \omega,
  \end{align*}
  provided that $H^1(M,\bR) = 0$.
\end{lemma}

\begin{lemma}\label{lemma4}
  For any smooth vector field $V$ on $\T$, we have that
  \begin{equation}
    2(V'[\Ric])^{1,1} =  \partial (\delta(G(V)) \omega).
  \end{equation}
\end{lemma}

Let us here recall how Lemma \ref{Vriccipot} is derived from Lemma
\ref{lemma4}. By the definition of the Ricci potential
\[\Ric = \Ric^H + 2 i \partial \bar \partial F,\]
where $\Ric^H = n\omega$ by the assumption $c_1(M, \omega) =
n[\frac{\omega}{2\pi}]$. Hence
\[V'[\Ric] = - d V'[I] d F + 2 i d \bar \partial V'[F],\] and
therefore
\[ 4i \partial \bar \partial V'[F] = 2(V'[\Ric])^{1,1} + 2\partial
V'[I] d F.\] From the above, we conclude that
$$ (2 (G(V) d F) \omega + \delta(G(V)) \omega -
4i \bar \partial V'[F])_\sigma \in \Omega^{0,1}_\s(M)$$ is a
$\partial_\sigma$-closed one-form on $M$. From Lemma \ref{dbarl}, it
follows that it is also $\bar \partial_\sigma$-closed, hence it must
be a closed one-form. Since we assume that $H^1(M,\bR) = 0$, we see
that it must be exact. But then it in fact vanishes since it is of
type $(0,1)$ on $M_\s$.

From the above we conclude that
$$
u(V) = \frac1{k+n/2}o(V) - V'[F] = - \frac1{4k+2n} \bigl(
  \Delta_{G(V)} + 2\nabla_{G(V)dF} + 4k V'[F]\bigr)
$$
solves \eqref{eqcond}. Thus we have established Theorem \ref{HCE} and
hence also provided an alternative proof of Theorem \ref{MainGHCI}.

In \cite{AGL} we use half-forms and the metaplectic correction to
prove the existence of a Hitchin connection in the context of
half-form quantization. The assumption that the first Chern class of
$(M,\omega)$ is $n [\frac{\omega}{2\pi}]\in H^2(M,\bZ)$ is then 
replaced by the vanishing of the second Stiefel-Whitney class of $M$
(see \cite{AGL} for more details).

Suppose $\Gamma$ is a group which acts by bundle automorphisms of $\L$
over $M$ preserving both the Hermitian structure and the connection in
$\L$. Then there is an induced action of $\Gamma$ on $(M,\omega)$. We
will further assume that $\Gamma$ acts on $\T$ and that $I$ is
$\Gamma$-equivariant. In this case we immediately get the following
invariance.

\begin{lemma}
  The natural induced action of $\Gamma$ on $\cH^{(k)}$ preserves the
  subbundle $H^{(k)}$ and the Hitchin connection.
\end{lemma}

\begin{remark}
\label{rem2}
We remark that if $M$ is not compact, but we know there exist a family of 
functions $F : \T \to C^\infty(M)$ which solves
\[
 \Ric = n\omega + 2i\partial \dbar F,
\]
then all the rest of the proof of Theorem \ref{HCE} is local and thus it applies 
in the noncompact case as well, and the theorem remains valid in this more 
general case.  
We further observe from the above argument that if the family $I$ satisfies 
$F_\sigma = 0$ for all $\sigma\in \T$, then the above construction also gives 
a Hitchin connection, which in that case is simply given by 
\[
 u(V) = -\frac 1 {4k} \Delta_{G(V)}.
\]
An example of this is if $M$ is a torus and $I$ is a family of linear complex 
structures on $M$, see e.g. \cite{AB}.  
\end{remark}

\section{Non-negative polarizations on moduli spaces}
\label{sect3}
In this section we review the setting and results from \cite{A3} and we discuss the immediate generalizations to surfaces with marked points.

Let $\Sigma$ be a closed oriented surface and let $R$ be a finite set of points on $\Sigma$ and set $\tilde{\Sigma} = \Sigma - R$.
\begin{defn}
  A system $\tilde{P}$ of $q$ disjoint closed curves on $\tilde{\Sigma}$ is called \emph{admissible} if no two curves from the system are homotopic on $\tilde{\Sigma}$ and none of the curves are null-homotopic on $\tilde{\Sigma}$  nor homotopic on $\tilde{\Sigma}$ to a curve which is contained in a disc-neighborhood of one of the points in $R$.
\end{defn}
For an admissible system of curves $\tilde{P}$ on $\tilde{\Sigma}$, let $\bar{\Sigma}$ be the complement in $\tilde{\Sigma}$ of the curves in $\tilde{P}$. Suppose $c_0$ is any assignment of conjugacy classes of $\SU(2)$ to each of the points in $R$. Let $N$ be the moduli space of flat $\SU(2)$-connections on $\tilde{\Sigma}$ with holonomy around each of the points in $R$ contained in the conjugacy classes determined by $c_0$.

We let $h_{\tilde{P}} : N \to [-2,2]^{\tilde{P}}$ be the map which maps a connection to the trace of the holonomy around the curves in $\tilde{P}$. We let $N_c = h^{-1}_{\tilde{P}}(c)$ for all $c \in [-2,2]^{\tilde{P}}$. Consider the moduli space $\bar{N}$ consisting of flat connections on $\bar{\Sigma}$ with holonomy around each of the points in $R$ contained in the conjugacy class determined by $c_0$. Let $\bar{N}^c$ be the subspace of $\bar{N}$ consisting of the connections, which also has holonomy around each of the two boundary components corresponding to any curve $\gamma \in \tilde{P}$ given by $c(\gamma)$. The projection map
\begin{align*}
	\pi : N \to \bar{N}
\end{align*}
induces projection maps
\begin{align*}
	\pi_c : N_c \to \bar{N}^c.
\end{align*}
For each of the conjugacy classes $c(\gamma) \in [-2,2]$, $\gamma \in \tilde P$, we choose an element in the conjugacy class and let $Z_{c(\gamma)}$ be the centralizer of this element in $c(\gamma)$. Hence we see that for $c(\gamma) \in (-2,2)$, we have that $Z_{c(\gamma)} \isom \U(1)$, and for $c(\gamma) = \pm 2$, we have that $Z_{c(\gamma)} \isom \SU(2)$. Furthermore, if we have a flat connection $\bar{A}$ on $\bar{\Sigma}$, representing a point in $\bar{N}^c$, we define $Z_{\bar A}$ to be the automorphism group of $\bar{A}$. Now fix a flat connection $A$ on $\tilde{\Sigma}$ such that $[A] \in N_c$ and $\pi_c([A]) = [\bar{A}]$. If we fix parametrizations of each of the components of a tubular neighborhood of $\tilde{P}$ by $S^1 \times (-1,1)$, which for each $\gamma \in \tilde P$ maps $S^1 \times \{0\}$ to $\gamma$, we can assume that $A$ restricted to each component of this tubular neighborhood is of the form $A = \xi_\gamma \, d\theta$, where $\theta$ is a coordinate on $S^1$ and $\xi_\gamma \in \su(2)$ such that $\exp(\xi_\gamma) \in c(\gamma)$ is the chosen element in the conjugacy class for all $\gamma \in \tilde{P}$. We can now associate to any element in the Lie group
\begin{align*}
	z \in Z_c = \prod_{\gamma \in \tilde{P}} Z_{c(\gamma)}
\end{align*}
a broken gauge transformation $g_z$ with support in the chosen tubular neighborhood of $\tilde{P}$, such that the restriction of $g_z$ to the connected component around $\gamma$ is given by $g = \exp(\psi(t)\eta(\gamma))$, where $z(\gamma) = \exp(\eta(\gamma))$ and $\psi : (-1,1) \to [0,1]$ is identically zero on $(0,1)$ and near $-1$, and it is identically $1$ on $(-\eps,0]$, for some small positive $\eps$.

From this, it is clear that the Lie group $Z_{\bar{A}}$ acts on $Z_c$ and we have the following Lemma from \cite{A3}.
\begin{lem}
  We have a smooth $Z_{\bar{A}}$-invariant surjective map
  \begin{align*}
		\tilde{\Phi}_A : Z_c \to \pi_c^{-1}([\bar{A}]),
  \end{align*}
  given by mapping $g \in Z_c$ to $g^*A$. This map induces an isomorphism
  \begin{align*}
		\Phi_A : Z_c/Z_{\bar A} \to \pi_c^{-1}([\bar{A}]).
  \end{align*}
\end{lem}
We observe that $Z_{\bar A}$ is isomorphic to a product of Lie groups. The product is index the components of $\bar{\Sigma}$ and the Lie groups are of sub-groups of $\SU(2)$ from the following list: $\SU(2)$, $Z_{\SU(2)} = \{ \pm \Id \}$, or a conjugate of $\U(1) \subset \SU(2)$.

We denote by $\calL_N$ the Chern--Simons line bundle over $N$
constructed in \cite{Fr}. $\calL_N$ is a topological complex line bundle over $N$. Moreover, there is a well-defined notion of parallel transport in this bundle along any curve in $N$ which can be lifted to a piecewise $C^1$-curve of connections. Over the dense smooth part $N'$ of $N$, $\calL_N$ is equipped with a preferred Chern--Simons connection, whose parallel transport induces this parallel transport.

Since $\calL_N$ is constructed in \cite{Fr} on the space of connections on $\tilde{\Sigma}$ with holonomy contained in $c_0$, we see in fact that we get a well-defined line bundle $\calL_{N,A}$ over $Z_cA \isom Z_c$, with an induced action of $Z_{\bar{A}}$, with the property that there is a natural $Z_{\bar A}$-equivariant isomorphism from $\calL_{N,A}$ to $\tilde{\Phi}^*_A(\calL_N)$. From this we see that the restriction of smooth sections of $\calL^k_N$ to $\pi_c^{-1}([\bar{A}])$ gets pulled back by $\tilde{\Phi}_A$ to $Z_{\bar{A}}$-invariant smooth sections of the smooth bundle $\calL^k_{N,A}$ over $Z_c$. This gives us a means to use differential geometric techniques to study these restrictions, even though these fibers sometimes are singular.
\begin{defn}
  The Bohr--Sommerfeld set $B_k(\tilde{P})$ associated to $\tilde{P}$ on $\Sigma'$ is by definition the subset of $c$'s in $h_{\tilde{P}}(N) \subset [-2,2]^{\tilde{P}}$, for which the holonomy in $\calL_N^k|_{N_c}$ along the fibers of $\pi_c$ is trivial.
\end{defn}
We remark that if $c \in B_k(\tilde{P})$, there is a unique complex line $\calL_{c,k}$ over $\bar{N}^c$ and a preferred isomorphism
\begin{align*}
		\pi_c^*(\calL_{c,k}) \isom \calL_N^k|_{N_c}.
\end{align*}
Let $\bar{\sigma}$ be a complex structure on $\bar{\Sigma}$ with the following property:
\begin{enumerate}
  \item The complex structure $\bar{\sigma}$ restricted to each of the components of a tubular neighborhood of the curves in $\tilde{P}$ are conformally equivalent to semi-infinite cylinders.
  \item The complex structure $\bar{\sigma}$ extends over the points in $R$.
\end{enumerate}
The following theorem is an immediate generalization of Theorem~5.1 in \cite{A3}.
\begin{thm}
  The structure $(\tilde{P},\bar{\sigma})$ induces a non-negative polarization $F_{\tilde{P},\bar{\sigma}}$ on $N$, with the following properties:
  \begin{itemize}
    \item The coisotropic leaves of $F_{\tilde{P},\bar{\sigma}}$ are given by the fibers $N_c$, $c \in [-2,2]^{\tilde{P}}$.
    \item The isotropic leaves of $F_{\tilde{P},\bar{\sigma}}$ in $N_c$ are fibers of $\pi_c : N_c \to \bar{N}^c$, for all $c \in [-2,2]^{\tilde{P}}$.
  \end{itemize}
\end{thm}
\begin{defn}
  Let $H^{(k)}_{\tilde{P},\bar{\sigma}}$ denote the vector space of distributional sections of $\calL_N^k$ over $N$, which are covariant constant along the directions of $F_{\tilde{P},\bar{\sigma}}$.
\end{defn}
We have the following factorization theorem, which is an analogue of the factorization theorem in \cite{A2}.
\begin{thm}
  We have the following natural isomorphism:
  \begin{align*}
		H_{\tilde P, \bar \sigma}^{(k)} \simeq \bigoplus_{c \in B_k(\tilde P)} H^0(\bar{N}^c_{\bar \sigma},\calL_{c,k}).
  \end{align*}
\end{thm}
\begin{proof}
   The theorem follows directly from the arguments presented in \cite{A2}. First one observes that for any $c \in h_{\tilde P}(N)$, the holonomy is trivial along some generic fiber of $\pi_c$ if and only if it is trivial along all the generic fibers of $\pi_c$. This follows since the symplectic annihilator of $TN_c$ is $\ker(\pi_c)_*$ at a generic point of $N_c$. From this one concludes that the support of any distribution in $H^{(k)}_{\tilde P,\bar \sigma}$ must be contained in $h^{-1}_{\tilde P}(B_k(\tilde{P}))$. For each $c$ in $B_k(\tilde{P})$ one then observes that a distribution in $H^{(k)}_{\tilde P, \bar \sigma}$ can be restricted to $N_c$, and here it must be covariant constant along the fibers of $\pi_c$ and hence induces a section in $\calL_{c,k}$ over $\bar{N}^c$. By analyzing the distributional section restricted to $N_c$ in the transverse directions to the fibers of $\pi_c$, one finds that the induced section of $\calL_{c,k}$ over $\bar{N}^c$ must be holomorphic with respect to the complex structure induced on $\bar{N}^c$ by $\bar{\sigma}$.
\end{proof}
Suppose we now have a complex structure $\tilde{\sigma}_0$ on $\tilde{\Sigma}$, which extends over $\Sigma$. We now construct a family of complex structures $\tilde{\sigma}_t$ on $\tilde{\Sigma}$, obtained from $\tilde{\sigma}_0$ by cutting $\tilde{\Sigma}$ along each of the curves in $\tilde{P}$ and gluing in flat cylinders of length $t$ to each of the two copies of each curve in $\tilde{P}$, for all non-negative $t$. The complex structures $\tilde{\sigma}_t$ on $\tilde{\Sigma}$ induce complex structures on $N$. When identifying the surface we obtain by cutting $\Sigma$ along $\tilde{P}$ and then attaching semi-infinite flat cylinders to all boundary components, with $\bar \Sigma$, we obtain a complex structure on $\bar{\Sigma}$, which we denote $\bar{\sigma}$.

The following theorem is an immediate generalization of Theorem~6.2 of \cite{A2}.

\begin{thm}
  \label{thm10}
  The complex structures on $N$ induced from the complex structures $\tilde{\sigma}_t$ converge to the non-negative polarization $F_{\tilde P, \bar \sigma}$ as $t$ goes to infinity.
\end{thm}

\section{The asymptotics of the Hitchin connection under degenerations}
\label{sect4}
In this section we prove Theorem~\ref{thm3}. We consider the more
general setting discussed in Theorem~\ref{thm10} from the previous
section. However, we only need the following special cases:
\begin{enumerate}
  \item The surface $\Sigma$ is of genus $g > 1$ and $R$ consists of one point,
  \item The surface $\Sigma$ is a torus and $R$ consists of one point,
  \item The surface $\Sigma$ is a sphere and $R$ consists of four points.
\end{enumerate}
We recall that the moduli space $N$ of interest is the moduli space of flat connections on $\tilde \Sigma$ with holonomy around each of the points in $R$ determined by $c_0$.

In the case (1) we will only be interested in the moduli space $N=M$
of flat connections on $\tilde \Sigma$ with holonomy $c_0=\{-\Id\}$
around the one point $p$ in $R$. Consider a point $\sigma \in \T$. 

A holomorphic vector bundle $E \to \tilde\Sigma_{\tilde\s}$ is
semi-stable if for every proper holomorphic subbundle $F \subset E$ we
have the following conditions on the \emph{slope} $\mu$ of $E$, and $F$
\[
\frac{\deg(F)}{\rk(F)} = \mu(F) \leq \mu(E) = \frac{\deg(E)}{\rk(E)}. 
\] A holomorphic vector bundle is called stable if the inequality is
strict.

To each semi-stable vector bundle there exists a unique (up to
isomorphism) filtration called the Jordan-H\"{o}lder filtration
\[
 0 = E_0 \subset \dots \subset E_m = E,
\]
with the property that the slopes of each of the quotients is the same
as the slope of $E$, i.e.
\[
\mu(E_{i+1}/E_i) = \mu(E),
\]
and each quotient $E_{i+1}/E_i$ is a stable vector bundle. We then
define the associated grated vector bundle
\[
\Gr(E) = \bigoplus_i (E_{i+1}/E_i).
\]
Two holomorphic vector bundles $E$, $E'$ are S-equivalent if and only
if their associated grated vector bundles are isomorphic, i.e.
\[
E \sim_S E' \quad \text{if and only if} \quad \Gr(E) \simeq \Gr(E').
\]

\begin{thm}[Narasimhan \& Seshadri]
  The moduli space of S-equivalence classes of semi-stable bundles of
  rank $n$ and determinant $\mathcal{O}_\s([p])$ is a smooth complex
  algebraic projective variety isomorphic as a K\"{a}hler manifold to $M_\s$
\end{thm}
This theorem is proven by using Mumford's Geometric
Invariant Theory. 

Hence we see that $\T$ parametrizes complex structures which are all K\"{a}hler with respect to the symplectic structure $\omega$ on $M$. To get uniform notation we will in this case (1) also use the notation $\tilde \T$ for $\T$.

In the cases (2) and (3) we are interested in arbitrary rational holonomies around the points in $R$, hence we need on the algebraic side to consider moduli space of parabolic vector
bundles on $\Sigma$ with the parabolic structures located at the points $R$ with respect to some point $\tilde \sigma$ in the Teichm\"{u}ller space $\tilde \T$ of $\tilde \Sigma$.

\begin{defn}
  Let $\tilde\Sigma_{\tilde\s}$ be a compact Riemann surface with distinct marked
  points $R \subset \tilde\Sigma_{\tilde\s}$, and $E \to \tilde\Sigma_{\tilde\s}$ a holomorphic vector
  bundle of rank $r$. A \emph{parabolic structure} on
  $E \to \tilde\Sigma_{\tilde\s}$ at $p \in S$ is a choice of partial flag
  \[
  E_{p} = E^1_{p} \supset E^2_{p} \supset \dots \supset
  E_{p}^{r(p)} \supset 0
  \]
  with a set of parabolic weights
  \[
  w_1(p) < \dots < w_{r(p)}(p), \quad \text{with} \quad w_{r(p)}(p)
  - w_1(p) < 1.
  \]
  Multiplicities are denoted by $m_j(p) = \dim E_{p}^j - \dim
  E_{p}^{j+1}$.

  A \emph{parabolic vector bundle} on $\tilde\Sigma_{\tilde\s}$ is a holomorphic rank
  $r$ vector bundle $E \to \tilde\Sigma_{\tilde\s}$ with a choice of parabolic
  structure at each marked point.
\end{defn}

In order for the moduli space of parabolic vector bundles to have nice
geometric structure we need to impose stability conditions on the
parabolic vector bundles -- just as in the case of ordinary vector
bundles.

  The \emph{parabolic degree} of a parabolic vector bundle $E \to
  \tilde\Sigma_{\tilde\s}$ is defined by
  \[
  \pdeg(E) = \deg(E) + \sum_{p \in R} \sum_{i} m_i(p)w_i(p).
  \]
The parabolic slope of $E$ is
\[
\mu(E) = \pdeg(E)/\rk(E).
\]
Every holomorphic subbundle $F$ of $E$ naturally has a parabolic
structure at each of the marked points $p \in R$ by defining
\[
F_p\cap E^1_p \supset F_p \cap E^2_p \supset \dots \supset F_p \cap
E_p^{r(p)} \supset 0,
\]
and removing repeated terms. The weights are the largest of the
corresponding parabolic weights from $E$, i.e $w_i^F(p) =
\max_{j}\{w_j \, | \, F_p \cap E_{p}^j = F_p^j\}$.

As with vector bundles we now define \emph{stable} parabolic vector
bundles as those where for each proper subbundle $F \subset E$ we have
\[
\mu(F) = \frac{\pdeg(F)}{\rk(F)} < \frac{\pdeg(E)}{\rk(E)} = \mu(E).
\]

The weights give the connection between the moduli space of parabolic
vector bundles to the moduli space of flat
unitary connections with holonomy around the punctured marked points
being these weights. This is the Mehta--Seshadri theorem \cite{MeSe}.

\begin{thm}[Mehta--Seshadri]
  Let $\tilde\Sigma_{\tilde\s}$ be a surface as above and $R \subset \tilde\Sigma_{\tilde\s}$ a set of marked points of $\tilde\Sigma_{\tilde\s}$. Then there
  is a one-to-one correspondence between the moduli space of
  irreducible unitary connections on $\tilde\Sigma_{\tilde\s} - R$ with
  holonomy around $p \in R$ having eigenvalues
  \[ \{ e^{2\pi i w_1(p)}, e^{2 \pi i w_2(p)},
    \dots, e^{2\pi i w_{r(p)}(p)}\},
    \]
  each $e^{2\pi i w_i(p)}$ with multiplicity $m_i(p)$, and the moduli space
  of parabolic vector bundles with parabolic degree zero on $\tilde\Sigma_{\tilde\s}$
  with weights and multiplicities specified by the above data.
\end{thm}

We will in the following only be interested in the case of
$\SU(2)$-connections corresponding to rank-$2$ degree $0$ parabolic vector
bundles. Furthermore we will only be interested in the cases where the
Riemann surface is a torus with a single marked point and the sphere
with four marked points.

At the marked points for a rank-$2$ parabolic vector bundle there is
only a two step filtration,
\[
E = E^1_p  \supset E^2_p \supset 0,
\]
for $p \in R$. The weights must satisfy $w_2(p) - w_1(p)<1$ and
$w_1(p) < w_2(p)$. If the parabolic vector bundle should correspond to
a flat unitary connection the parabolic degree of $E$ must be zero, so
\[
0 = \pdeg(E) = \deg(E) + \sum_{p \in S} w_1(p) + w_2(p).
\] 
At each marked point $p \in R$ the holonomy of the connection around
that point is conjugate to $\diag(e^{2 \pi i w_1(p)}, e^{2 \pi i
  w_2(p)})$. Since this matrix must be an $\SU(2)$ matrix
$w_1(p)+w_2(p)$ must be an integer. Since $\deg(E) = 0$ we all in all
have $w_i(p) \in(-\frac{1}{2},\frac{1}{2})$. The consequence is that $w_1(p) + w_2(p) = 0$ and finally that $w_1(p)
= -w_2(p)$. Since $w_1(p) < w_2(p)$ we get that $w_2(p) = s_p \in
[0,\frac{1}{2})$ and $w_1(p) = -s_p \in (-\frac{1}{2},0]$. 

Let $L$ be a proper line subbundle of $E \to \tilde\Sigma_{\tilde\s}$. If we assume $E$
to be parabolically stable then $\pdeg(L) < 0$. For a marked point $p$ the
filtration of $L_p$ has only one step, and is
\[
L_p = L_p \cap E_p^1 \supset L_p \cap E^2_p = \begin{cases} 0 & L \neq
  E^2_p \\ L_p & L_p = E^2_p \end{cases}
\]
In the case $L_p \neq E^2_p$ the weight is $w_1(p) = -s_p$ while if
$L_p = E^2_p$ the weight jumps to $w_1(p) = s_p$. 

In all of the three cases (1) - (3) above, we get a family of complex structures $I$ on $N$ parametrized by $\tilde{\calT}$. We denote $N$ with the complex structure $I(\tilde \sigma)$ by $N_{\tilde \sigma}$ for $\tilde \sigma \in \tilde \T$. We let $H^{(k)}$ denote the vector bundle over $\tilde \calT$, whose fiber over $\tilde \sigma \in \tilde \calT$ is $H^0(N_{\tilde \sigma},\calL^k_N)$.

\begin{lem}
  In the cases (1)---(3) above, we have that $N$ and $I$ satisfy either the assumptions of Theorem~\ref{HCE} or those of Remark~\ref{rem2}, hence in all cases we have a Hitchin connection which is projectively flat.
\end{lem}
\proof
In the case (1) this was demonstrated by Htichin in \cite{H}. The cases (2) and (3) follow from the special considerations in Sections \ref{4p} and \ref{elliptic}.
\eproof
Consider the family $\tilde{\sigma}_t$ constructed in the previous section from the starting data $(\tilde \sigma_0, \tilde P)$. Let
\begin{align*}
	P_t(\tilde \sigma_0, \tilde P) : H_{\tilde \sigma_0}^{(k)} \to H_{\tilde \sigma_t}^{(k)}
\end{align*}
be the parallel transport with respect to the Hitchin connection in $H^{(k)}$ over $\tilde \calT$ along the curve $(\tilde \sigma_s)$, $s \in [0,t]$.

Let $c \in [-2,2]^{\tilde P}$ and consider the subspace $N_c \subset N$. Consider a point $x$ in $N'$ ($N'$ being the manifold of smooth points of N), which is also a smooth point of $N_c$. For each $t$, let $I_t$ be the corresponding complex structure on $N$. A covariant constant section $s_t \in H_{\tilde \sigma_t}^{(k)}$, $t \in [0,\infty)$, of the Hitchin connection along the curve $\tilde \sigma_t$ satisfies the following equations:
\begin{align*}
	s_t' = u(\tilde \sigma_t')(s_t),
\end{align*}
and
\begin{align*}
	\nabla_X s_t = -i\nabla_{I_t X}s_t
\end{align*}
for all vector fields $X$ and all $t$. Since the curves in $\tilde P$ are non-intersecting, the corresponding holonomy functions Poisson commute, hence we have that $TN_c$ is coisotropic, thus $TN_c^0 \subset TN_c$, where $(\cdot )^0$ refers to the symplectic complement. We observe that $TN_c^{\perp_t} = I_t(TN_c^0)$, where $(\cdot)^{\perp_t}$ refers to the orthogonal complement with respect to the metric induced by $\omega$ and $I_t$. From this we get the following decomposition:
\begin{align}
  \label{eq9}
	TN|_{N_c} = TN_c \oplus I_t(TN_c^0).
\end{align}
For any section $X$ of $TN|_{N_c}$, we define $X'$ a section of $TN_c^0$ and $X''$ a section of $I_t(TN_c^0)$ such that $X = X' + X''$.
\begin{thm}
  \label{prop1}
  Suppose $s_0 \in H_{\sigma_0}^{(k)}$. Then $s_t|_{N_c}$ only depends on $s_0|_{N_c}$ and we have that
  \begin{align}
    \label{eq10}
		(s_t|_{N_c})' = \tilde{u}_c(\tilde \sigma_t')(s_t|_{N_c}),
  \end{align}
  where $\tilde u_c(\tilde \sigma_t')$ is a second order differential operator acting on $C^\infty(N_c, \L^k_N|_{N_c})$ depending linearly on $\tilde \sigma_t'$. Moreover, the limit
  \begin{align}
		\label{eq11}
		\tilde u_{c,\infty} = \lim_{t \to \infty} \tilde u_c(\tilde \sigma_t')
  \end{align}
  exists, and the operator $\tilde u_{c,\infty}$ is a second order differential operator acting on sections of $\calL^k_N|_{N_c}$, whose kernel consists of sections of $\calL^k_N|_{N_c}$ that are covariant constant along the directions of $F_{\tilde P, \bar \sigma} \cap \bar F_{\tilde P, \bar \sigma}$.
\end{thm}
We will use the following notation
$$\tilde{u}_{c,t} = \tilde{u}_c(\tilde \sigma_t').$$
\begin{proof}
  For $X$ a smooth section of $TN|_{N_c'}$, we have that
  \begin{align*}
		\nabla_X = \nabla_{X'} - i\nabla_{I(X'')},
  \end{align*}
  and if $Y$ is a further smooth section of $TN|_{N_c'}$, then
  \begin{align*}
		\nabla_X \nabla_Y =&\, \nabla_{X'}\nabla_{Y'-iI_t(Y'')}-i\nabla_{Y'}\nabla_{I_t(X'')}+i\nabla_{I_t(Y'')}\nabla_{I_t(X'')} \\
		 &\, + \nabla_{[X'',Y']-iI_t([X'',I_t(Y'')]'')} \\
		 &\, - \nabla_{[X'',I_t(Y'')]'-iI_t([X'',I_t(Y'')]'')} \\
		 &\, -k(i\omega(X'',Y') - \omega(X'',I_t(Y''))).
  \end{align*}
  From these formulae we immediate get the first part of the proposition, since we can use the above two formulae to rewrite $u(\tilde \sigma')|_{N_c}$ to obtain an operator $\tilde u_c(\tilde \sigma_t')$, such that the evolution of $s_t|_{N_c}$ is determined by \eqref{eq10}.
  
  Let us now use the notation $G_t = G(\tilde \sigma_t')$.
  
  \begin{claim}\label{C1}
    There exists a unique section $G_\infty \in C^\infty(N'_c,S^2(F_{\tilde P, \bar \sigma} \cap \bar F_{\tilde P, \bar \sigma}))$ such that
    \begin{align*}
		  \lim_{t \to \infty} G_t = G_\infty.
    \end{align*}
  \end{claim}
  In order to establish the claim, we consider a point $x_0 \in N_c'$ and a local symplectic frame $(w,v)$ of $TN'$ around $x_0$ with the following properties: The bundles $O = \mathrm{Span} \ p$ and $Q = \mathrm{Span} \ q$ are complementary Lagrangian subbundles of $TN'$ and further that $p = (p',p'')$, such that
  \begin{align*}
		\mathrm{Span} \ p' = F_{\tilde P, \bar \sigma} \cap \bar F_{\tilde P, \bar \sigma} \cap TN'.
  \end{align*}
  We now observe that there is a unique complex symmetric matrix $Z_t(x)$ depending smoothly on $x$ near $x_0$, such that
  \begin{align*}
		w^{(t)}(x) = p(x) + Z_t(x) q(x)
  \end{align*}
  spans $P_t(x)$, the fiber of the holomorphic tangent bundle of $N'$ at $x$ with respect to the complex structure induced from $\tilde \sigma_t$. If we write $Z_t = X_t + iY_t$, where $X_t$ and $Y_t$ are real, then from its definition we conclude that $X_t$ and $Y_t$ are symmetric and $Y_t > 0$. The decomposition $p = (p',p'')$ gives a corresponding decomposition of $q = (q',q'')$. This decomposition gives the following block-decomposition of $Z_t$:
  \begin{align*}
		Z_t = \begin{pmatrix} Z_t^{(11)} & Z_t^{(12)} \\ Z_t^{(21)} & Z_t^{(22)} \end{pmatrix}.
  \end{align*}
  By Theorem~\ref{thm10}, we have the following asymptotics:
  \begin{align*}
		Z_t \to \begin{pmatrix} 0 & 0 \\ 0 & Z_\infty \end{pmatrix}
  \end{align*}
  as $t$ goes to infinity, where $Z_\infty = X_\infty + iY_\infty$ and $Y_\infty > 0$. By examining the proofs of Theorem~10 in \cite{A3}, one sees immediately that the convergence of $P_t$ to $F_{\tilde P , \bar \sigma}$ is a convergence in the $C^\infty$-topology on $N'$. In particular, we have that
  \begin{align*}
		Z_t = Z_\infty + Z_\infty' t^{-1} + R(t).
  \end{align*}
  
 Let us now analyse the case where 
   \begin{align*}
		O = F_{\tilde P, \bar \sigma}  = F_{\tilde P, \bar \sigma} \cap \bar F_{\tilde P, \bar \sigma} \cap TN'.
  \end{align*}
  The other cases are treated completely analogously.
  
Let $L_t$ be a symplectic local bundle 
transformations of $TN'\otimes {\bC}$ such that  $L_t(O)=P_t$ and $L_\infty = \Id$.   
In this basis we have:
\[L_t = \begin{pmatrix}A & B \\ 
C & D\end{pmatrix} \to \begin{pmatrix}\Id & 0 \\
0 & \Id \end{pmatrix}
\]
as $t\to \infty$.
Since $A \to \Id$ as $t\to \infty$, we may assume that $A$ is invertible.   
The symplectic transform
\[
\begin{pmatrix}
	A^{-1} & 0 \\ 
	-C^t & A^t
\end{pmatrix}
\]
preserves $O$ so we consider
\[
\begin{pmatrix}
	A^{-1} & 0 \\
	-C^t & A^t
\end{pmatrix}
\begin{pmatrix}
	A & B \\ 
	C & D
\end{pmatrix}
= 
\begin{pmatrix}
	\Id & A^{-1}B \\
	A^tC - C^tA & A^tD-C^tB
\end{pmatrix}
= 
\begin{pmatrix}
	\Id & A^{-1}B \\
	0 & \Id
\end{pmatrix}
\]
which must map $O$ onto $P_t$. Hence, $Z=A^{-1}B$ and 
\[
	w_i = p_i \sum Z_{ij} q_j = p_i + \sum X_{ij}q_j + i\sum Y_{ij}q_j
\]
and
\[
	\bar w_i = p_i \sum \bar Z_{ij} q_j = p_i + \sum X_{ij}q_j - i\sum y_{ij}q_j
\]
is a basis of $\bar P_t$ (we have here suppressed the $t$-dependence of the $w_i$'s).
Since $P_t \cap \bar P_t = \{0\}$ we that $O\cap P_t=\{0\}
	$. This follows since $P_t$ corresponds to  $I_t$  and $I_t(O)\cap O = \{0\}$ since $O$ is Lagrangian.  
	
	\begin{claim}
	$P\cap	P_t=\{0\} \Leftrightarrow \det(Z)\neq 0$.    
	
\end{claim}
\proof
	Assume 
	$\det{Z}\neq 0$. Then there exists a non zero vector $c$ such that 
	$$\sum_i c_iZ_{ij}=0, \qquad j=1, \dots, n.$$  Hence $\sum 
	c_iw_i = \sum c_i p_i$ thus  $O \cap P_t \neq \{0\}$.  
	Conversely, if 
	$O\cap P_t\neq \{0\}$, let $c$ be such that 
	$$\sum c_iw_i \in O\cap P_t-\{0\}.$$   But then 
	$$\sum c_iw_i = \sum c_ip_i + \sum_j (\sum_i c_iZ_{ij})q_j \in P\cap P_t-\{0\}$$
	which implies that $\sum_i c_iZ_{ij} = 0$ for  $j=1, \dots, n$., thus  $\det Z = 0$.  
	\eproof

\begin{claim}
	$P_t\cap \bar{P_t}=\{0\}$ if and only if $\det{Y}\neq 0$.
\end{claim}
\begin{proof}
	Assume $\det{Y}=0$ then there exist $(x_1,\ldots,x_n)\in \bR^n
        -\{0\}$ s.t.
	$\sum x_i Y_{ij}=0$. Now, $\sum x_i w_i= \sum x_i \bar w_i \neq 0$ and hence $P_t\cap \bar P_t \neq \{0\}$.
	
	Conversely, assume $P_t\cap \bar P_t \neq \{0\}$. Let $(c_1,\ldots,c_n)\in \bC^n - \{0\},$ such that 
	$\sum c_i w_i \in P_t \cap \bar P_t \cap TM$. But then 
	
	$$\sum(c_iw_i-\bar c_i \bar w_i)=0,$$
	if and only if
	$$\sum(c_i'w_i'-c_i''w_i''+i(c_i'w_i''+c_i''w_i')-c_i'w_i'+c_i''w_i''+i(c_i'w_i''+c_i''w_i'))=0$$
	which is equivalent to 
	$$ 2\sum \left( c_i'Y_{ij}q_j + c_i''p_i + c_i''X_{ij}q_j \right) = 0,$$
	happening if and only if
	$$ c_i'' = 0 \text{ and } c_i'=0$$
	 which is the case if and only if $\det{Y}=0.$
\end{proof}

Notice that $TM/O \simeq Q$ and so
	$$ \dot{L}_{\infty} \in C^{\infty}(O,TM/O)$$
	is represented by $\dot{Z}_{\infty}$ in the basis $(p_1,\ldots,p_n)$ of $O$ and $(q_1,\ldots,q_n)$ of $Q\simeq TM/O$.
	By Proposition 2.3 p.118 in \cite{GS}, we can identify the space of lagrangian subspaces transverse to a given one $O$ as an affine space associated to the vector space $S^2(TM/P)$, which we can identify with $S^2(Q)$.
The quadratic form associated with $P_t$ becomes 
	$$ H_t(q_1,y_z) = (\pi_t y_1,y_z), $$
	where $\pi_t$ is the projection from $TN'\otimes \bC$ onto $P_t$ along $O$.
	Now $w_i = \pi_t (w_i) = \sum Z_{ij}^{-1}w_j$, and so 
	$$ H_t(q_i,q_j)=Z_{ij}^{-1}.$$
	Now $I_t$ is determined from $P_t$ by the condition that 
	$$ P_t = E(I_t,i), $$
	and 
	$$ \bar P_t = E(I_t,-i).$$
	Hence
	$$ I(p_i)+\sum X_{ij}I(q_j)+i\sum Y_{ij}I(q_j) = i p_i + i\sum X_{ij}q_j - \sum Y_{ij}q_j $$
	and 
	$$ I(p_i) + \sum X_{ij}I(q_j) - i\sum Y_{ij} I(q_j) = -iq_i - i\sum X_{ij} q_j - \sum Y_{ij}q_j,$$
	which implies 
	$$ I(q_j) = \sum_{k} Y_{jk}^{-1}\left( p_k + \sum X_{ki}q_i \right) = \sum_k Y_{jk}^{-1}p_k + \sum_{k,i} Y_{jk}^{-1}X_{ki}q_i,$$
	and
\begin{align*}	
	 I(p_i) 	&= - \sum Y_{ij}q_j - \sum X_{ij} I(q_j) \\
	 			&= - \sum Y_{ij}q_j - \sum X_{ij}Y_{jk}^{-1}p_k - \sum X_{ij}Y_{jk}^{-1}X_{kl}q_l \\
				&= - \sum X_{ij}Y_{jk}^{-1}p_k - \sum \left( Y_il + \sum X_{ij}Y_{jk}^{-1}X_{kl} \right)q
\end{align*}
This gives us the following matrix presentation
\[
	\begin{pmatrix}
		I(p) \\
		I(q)
	\end{pmatrix}
= 
	\begin{pmatrix}
		-XY^{-1}	& - (Y+XY^{-1}X) \\
		Y^{-1}		& Y^{-1X}
	\end{pmatrix}
	\begin{pmatrix}
		p \\
		q
	\end{pmatrix}.
\]

A simple computational check shows that this this matrix indeed squares to $-\Id$. Let us now compute the derivative of $I_t$.
\[
	\begin{pmatrix}
	\dot I (p) \\
	\dot I (q) 
	\end{pmatrix}
=
	\begin{pmatrix}
	-\dot{X} Y^{-1} - X\dot{Y}^{-1}		& -(\dot{Y}-\dot{(XY^{-1}X)}) \\
	\dot{Y}^{-1} 						& \dot{Y}^{-1}X+Y^{-1}\dot{X}
	\end{pmatrix}
	\begin{pmatrix}
	p \\
	q
	\end{pmatrix}
\]
Using that $YY^{-1}=\id$ we find that
\[
 \dot{Y}^{-1}=Y^{-1}\dot{Y}Y^{-1}.
\]
%\[
		%\xymatrix{\dot{P}_t \in C^{\infty}\left( \Hom\left(\bar{P}_t ,TM/\bar{P}_t \right)  \right) } \ar[d] & \simeq &  C^{\infty}\left( \Hom\left(\bar{P}_t,P_t \right)  \right)\\
		%C^{\infty}\left( \Hom\left({P}_t ,TM/{P}_t \right)  \right) & \simeq & \Omega^{0,1}(T)
		
%\]
On the other hand, using
\[
	I(w_i)=iw_i
\]
we compute that
\[
	\dot I(w_i) = i \dot w_i - I(\dot w_i) = (i\Id-I)\dot w_i
\]
and that
\[
	\dot w_i = \dot X_{ij}q_j +i \dot Y_{ij} q_j = (\dot X_{ij} +i \dot Y_{ij})q_j = (\dot X + i \dot Y) Y^{-1} Y q_j
\]
which gives 
\begin{align*}
	I(\dot w_i) 	&= (\dot X_{ij}+i\dot Y_{ij})I(q_j) \\
					&= \dot X_{ij}Y_{jk}^{-1} p_k + i \dot Y_{ij}Y_{jk}^{-1}p_k \\
					&= \dot X_{ij}Y_{jk}^{-1} X_{kl} q_l + i \dot Y_{ij}Y_{jk}^{-1}X_{kl}q_l .
\end{align*}
This gives us the following formular
\begin{align*}
	\dot I(w) 	&= -(\dot X + i\dot Y) Y^{-1} p +(\dot X + i \dot Y)Y^{-1}(iY-X)q \\
					&= -(\dot X + i\dot Y) Y^{-1}(p+Xq-Yq)\\
					&= -(\dot X+i \dot Y) Y^{-1} \bar w	 \\
					&= -\dot Z Y^{-1}\bar w.
\end{align*}
But $I(w_i)=i w_i$, so $\dot I(w) = i \dot w -I(\dot w).$ Now
\[
	\dot w = \dot Z q = \dot ZY^{-1}Y q, 
\]
so we conclude
\[
	I(\dot w) = \dot Z Y^{-1}(p+Xq).
\]
Which implies that
\[
	\dot I(w) = -\dot Z Y^{-1} (p+Xq-iYq) = -\dot Z Y^{-1} \bar{w}.
\]
Hence, with respect to the local frames we have the local matrix presentations
\[
	\dot I = -\dot Z Y^{-1} \in C^{\infty}(P_t^* \otimes \bar P_t) \simeq C^{\infty}(\Hom(P_t,\bar P_t)) 
\]
and 
\[
	\dot I = -\dot{\bar Z} Y^{-1} \in C^{\infty}(\bar P^*_t \otimes P_t) \simeq C^{\infty}(\Hom(\bar P_t,P_t)).
\]
So we have the following formula for the derivative of the complex structure
\[
	\dot I = -\sum_{i,j,k} \dot{\bar Z}_{jk}Y^{-1}_{ki}w_i \otimes \bar{w_j}^* = \sum_{i,j} a_{ij}w_i \otimes \bar{w}_j^*.
\]
So if $w=\sum \omega_{ij} w_i^* \wedge \bar{w}_j^*$ then
\[
	a_{ij}=\sum G_{ik}w_{kj}
\]
so
\[
	\sum_k a_{ik} w^{-1}_{kj}=-\sum_{l,k}Y^{-1}_{ik}\dot{\bar Z}_{kl} \omega_{lj}^{-1}.
\]
Define
\[
	Z^{-1}=V+iW.
\]
Then
\[
	XV-YW=\id \quad YV+XW=0
\]
and hence
\[
	V=-Y^{-1}XW=-WXY^{-1}.
\]
\[
	XY^{-1}XW+YW=(XY^{-1}X+Y)W=\id.
\]
Now 
\[
	(XY^{-1}X+Y)v=0 
\]
will imply that $Yv=-XY^{-1}Xv $ which gives
\[0\leq (Yv,v)=-(XY^{-1}Xv,v)=-(Y^{-1}Xv,Xv)\leq 0.
\]
Thus $(Yv,v)=0$ and therefore $v=0$.
\[
	W=-(XY^{-1}X+Y)^{-1}.
\]
\[
	V=Y^{-1}X(XY^{-1}X+Y)^{-1}.
\]

Let $(p_i^*,q_i^*)$ be a basis of $T^*N'$ dual to the basis $(p_i,q_i)$ of $TN'$. So $\omega=\sum p_i^*\wedge q_i^*$. Let $(w_i^*,\bar w_i^*)$ be a basis of $T^*N'$ dual to the basis $(w_i,\bar w_i)$ of TN'. Then
\[
	O^*= span\{ p_1^*,\ldots,p_n^* \}
\]
and
\[
	\bar P_t^*= span\{ \bar w_1^*,\ldots, \bar w_n^* \}.
\]
A short computation gives that
\[	\begin{pmatrix}
		w^* \\
		\bar w^*
	\end{pmatrix}
	=	\frac{i}{2}
	\begin{pmatrix}
		Y^{-1}	& 0 \\
		0 		& Y^{-1}
	\end{pmatrix}
	\begin{pmatrix}
		\bar Z	& -\id \\
		-Z 		& \id
	\end{pmatrix}
	\begin{pmatrix}
	p^* \\
	q^*
	\end{pmatrix}.
\]
Let us now compute the symplectic form on the $(w^*,\bar w^*)$ basis.
\begin{align*}
	\omega 	&= \sum_{i=1}^n p_i^* \wedge q_i^* \\
			&= \sum_{\substack{i=1\\ j=1}}^n (w_i^* \bar w_i^* ) \wedge (Z_{ij}w_j^* + \bar Z_{ij} \bar w_j^*)\\
			&= \sum_{\substack{i=1\\ j=1}}^n w_i^* \wedge Z_{ij} w_j^* + \sum_{\substack{i=1\\ j=1}}^n \bar w_i^* \wedge  \bar Z_{ij} \bar w_j^*+ \sum_{\substack{i=1\\ j=1}}^n w_i^* \wedge  \bar Z_{ij} \bar w_j^* + \sum_{\substack{i=1\\ j=1}}^n \bar w_i^* \wedge Z_{ij} w_j^* \\
			&= \sum_{i<j}^n (Z_{ij}(w_i^* \wedge w_j^*+w_j^* \wedge w_i^*)+ \bar Z_{ij}(\bar w_i^* \wedge \bar w_j^*+\bar w_j^* \wedge \bar w_i^*)) -2i \sum_{i,j} w_i^* \wedge Y_{ij}\bar w_j^* \\
			&= -2i \sum_{i,j=1} w_i^* \wedge Y_{ij} \bar w_j^* = -2i w^* \wedge Y \bar w^* \\
			&= -2i \sum_{i,j} Y_{ij} w_i^* \wedge \bar w_j^*,
\end{align*}
hence $\omega_{ij}= -2i Y_{ij}.$
From this we see that 
\[
	G= \sum_{i,j} G_{ij} w_i \otimes w_j = -\frac{i}{2} \sum_{i,j,k,l} Y^{-1}_{ik} \dot{\bar Z} Y^{-1}_{lj} w_i \otimes w_j
\]
Let $\pi_t: T^*M \to P_t^*$ be the projection onto $P_t^*$, whose kernel is $\bar P_t^*$, i.e. compatible with $T^*M=P_t^* \oplus \bar P_t^*$, and let 
$\pi'_t: T^*M\to O^*$ be the projection onto $O^*$, whose kernel is $\bar P^*_t$, i.e. compatible with $T^*M=O^*\oplus \bar P^*_t$.
Since $\Im(1-\pi_t)=\Ker \pi_t  = \Ker \pi_t' =\Im(1-\pi_t')$
we see that 
\[
	\pi_t \circ \pi_t' = \pi_t(\pi_t'+(1-\pi_t'))=\pi_t
\]
and 
\[
	\pi_t'\circ\pi_t = \pi_t'(\pi_t+(1-\pi_t))= \pi_t'.
\]
Let us now compute $\pi_t$ and $\pi_t'$ in the respective bases. We have that
\[
	-2i Y(w^*+\bar w^*)=(\bar Z -Z)p^* = -2i Y p^*
\]
which implies 
\begin{equation}\label{star}
	p^* = w^*+\bar w^*
\end{equation}
and further that
\[
	-2i Y (w^* - \bar w^*)=-2 q^* +2Xp^*.
\]
This implies
\[
	q^*= i Y (w^* - \bar w^*)+X(w^* + \bar w^*) = Zw^*+\bar Z \bar w^*.
\]
So 
\[
	\pi_t(p^*)=w^* =\frac{i}{2}Y^{-1}(-q^*+\bar Z p*), \quad \pi_t(q^*)=Z w^* = \frac{i}{2}ZY^{-1}(-q^* + \bar Z p^*).
\]
Because of \eqref{star} we see that
\[
	\pi_t'(w^*)=p^*.
\]
and 
\[
	\pi_t'(\bar w^*) =0
\]
so
\[
	\pi_t'(q^*)=\pi_t'(Zw^*)+\pi_t'(\bar Z \bar w^*) = Zp^*.
\]
Let us define the following operators
$$D' = \pi'_t \circ \nabla^{1,0} : C^\infty(\L^k) \ra C^\infty(O^*\otimes \L^k)$$
$$ G' = \pi'_t \circ G \circ \pi_t : C^\infty(O^*\otimes \L^k) \ra C^\infty(O\otimes \L^k)$$
$$ D'' = \pi'_t \circ (\nabla^{1,0}\otimes \Id \oplus \Id \otimes \nabla^{1,0})  \circ \pi_t : C^\infty(O\otimes \L^k) \ra C^\infty(O^*\otimes O\otimes \L^k).$$
On $\ker \nabla^{0,1}$, we shall now compute $\tr(D''GD')$. Hence if we have a section $s$ of $\L^k$ over $N'$, which is holomorphic, we have that
\[
	 \nabla_{p_i}s=-\nabla_{\bar Z_{ij} q_j}s,
\]
which we will use a number of times below. From the above we have that
\[
	G(w_i^*)=\frac{i}{2}\sum Y_{ik}^{-1}\dot{ \bar Z}_{kl} Y^{-1}_{lj}w_j
\]
Now
\[
	w_i+\bar w_i = 2 \left(p_i \sum X_{ij}q_j \right)
\]
and
\[
	w_i-\bar w_i = 2 i \sum Y_{ij}q_j
\]
so 
\[
	q_i=-\frac{i}{2}\sum Y_{ij}^{-1} (w_j-\bar w_j)
\]
and therefore we have that
\[
	w_i+\bar w_i = 2p_i - i\sum X_{ij}Y^{-1}_{jk}(w_k-\bar w_k)
\]
which implies 
\[
	\sum_{k}\left( \delta_{ik}+ i\sum X_{ij}Y^{-1}_{jk} \right)w_k = 2 p_i + \sum_{k} \left( i\sum X_{ij}Y^{-1}_{jk} - \delta_{ik}\right).
\]
Now
\[
	i \bar Z Y^{-1}w= 2p + iZY^{-1} \bar {w}
\] which gives
\[
 w=-2iY\bar Z^{-1}p+Y\bar Z^{-1}ZY^{-1}\bar w
\]
hence
\[
	\pi_t'(w)=-2i Y\bar Z^{-1}p, \quad \pi_t(p)= \frac{i}{2}\bar Z Y^{-1} w.
\]
So
\[
	\nabla^{1,0}s = \sum_{i=1}^n w_i^* \otimes \nabla_{w_i}s = \sum_{i=1}^n w_i^* \otimes \left( \nabla_{p_i}s + \nabla_{Z_{ij}q_j}s \right).
\]
and then if we use $\nabla^{0,1}s=0$ then
\[
	\nabla^{1,0}s = \sum_{i=1}^n w_i^* \otimes \nabla_{w_i}s = -2i \sum_{i,k,l=1} Y_{ik}\bar Z_{kl}^{-1}\nabla_{pl}s \otimes w_i^*,
\]
and hence
\[
	D' s = \pi'_t \nabla^{1,0}s = -2i \sum_{i,k,l=1}^n Y_{il} \bar Z^{-1}_{lk} p_i^* \otimes \nabla_{p_k}s.
\]
Now 
\begin{align*}
	G'(p_i^*)	&= \pi_t' \circ G \circ \pi_t(p_i^*) \\
				&= \pi_t'\circ G(w_i^*) \\
				&= -\frac{i}{2}\pi_t' \left( \sum_{k,l} Y^{-1}_{ik}\dot{\bar Z}^{-1}_{kl} Y_{lj}^{-1} w_j\right) \\
				&= - \sum Y_{ik}^{-1} \dot{\bar Z}_{kl} Y^{-1}_{lj}Y_{jr}\bar Z_{rs}^{-1} p_s \\
				&= -\sum_{k,l,r} Y^{-1}_{ik} \dot{\bar Z}_{kl} \bar Z_{lr}^{-1}p_r.
\end{align*}
so
\begin{align*}
G' \circ D's 	&= 2i \sum Y_{il} \bar Z_{lk}^{-1} Y_{ir}^{-1}\dot{\bar Z}_{rs}\bar Z_{st}^{-1} p_t \otimes \nabla_{p_k}s \\
				&= 2i \sum \bar Z_{lk}^{-1} \dot{ \bar Z}_{ls} \bar Z_{st}^{-1} p_t \otimes \nabla_{p_k} s \\
				&= 2i \sum \bar Z_{jl}^{-1} \dot{\bar Z}_{lk} \bar Z_{ki}^{-1} p_i \otimes \nabla_{p_j} s
\end{align*}
giving
\begin{align*}
	\pi_t \circ G' \circ D's 	&= -\sum \bar Z_{jl}^{-1}\dot{\bar Z}_{lk} \bar Z_{ki}^{-1} \bar Z_{ir} Y_{rs}^{-1} w_s \otimes \nabla_{p_j} s\\
								&= -\sum \bar Z_{jl}^{-1}\dot{\bar Z}_{lk} Y_{ki}^{-1} w_i \otimes \nabla_{p_j} s \\
								&= G\circ \nabla^{1,0}s.
\end{align*}
and thus
\begin{align*}
	\left( \nabla^{1,0}\otimes 1 + 1 \otimes \nabla^{1,0} \right) \circ G \circ \nabla^{1,0}s 
	= \,	& -\sum \bar Z_{jr}^{-1}\dot{\bar Z}_{rk} Y^{-1}_{ki} w_l^* \otimes \nabla_{w_l}(w_i) \otimes \nabla_{p_j}s \\
			& -\sum \bar Z_{jr}^{-1}\dot{\bar Z}_{rk} Y^{-1}_{ki}w_l^* \otimes w_i \nabla_{w_l} \nabla_{p_j} s \\
			& -\sum d(\bar Z_{jr}^{-1}\dot{\bar Z}_{rk} Y^{-1}_{ki})(w_l) w_l^* \otimes w_i \otimes \nabla_{p_j} s.
\end{align*}
Write
\[
	\nabla_{w_l}(w_i)=\sum C_{j}^{l,i}w_j
\]
Also, we rewrite
\begin{align*}
	\nabla_{w_l} \nabla_{p_j} s 		
		&=	-2i \sum_{k,m} Y_{l,k} \bar Z_{km}^{-1} \nabla_{p_m} \nabla_{p_j}s + \sum_{k,m,l,r}Y_{lk}\bar Z_{km}^{-1}Z_{ml}Y_{lr}^{-1} \nabla_{\bar w_r}		 	\nabla_{p_j}s \\
		&=	-2i \sum_{k,m} Y_{l,k} \bar Z_{km}^{-1} \nabla_{p_m} \nabla_{p_j}s + k\sum_{k,m,l,r}Y_{lk}\bar Z_{km}^{-1}Z_{ms}Y_{sr}^{-1} \omega(\bar w_r, 			p_j) s.
\end{align*}
which allows us to conclude
\begin{align*}
	\left( \nabla^{1,0}\otimes 1 + 1 \otimes \nabla^{1,0} \right) \circ G \circ \nabla^{1,0} s 
		= \,	& 	2i \sum \bar Z_{jr}^{-1} \dot{\bar Z}_{rk} Y_{ki}^{-1} Y_{ls} \bar Z_{sm}^{-1} w_l^* \otimes w_i \otimes \nabla_{p_m}\nabla_{p_i} s\\				& 	-  \sum \bar Z_{jr}^{-1} \dot{\bar Z}_{rk} Y_{ki}^{-1} C_{s}^{l,i} w_l^* \otimes w_s \otimes \nabla_{p_j} s \\
				& 	- k \sum \bar Z_{jr}^{-1} \dot{\bar Z}_{rk} Y_{ki}^{-1} Y_{ls} \bar Z_{st}^{-1} Z_{tm} Y_{mn}^{-1} \omega(\bar w_n,p_j) w_l^*\otimes 					w_i \otimes s \\
				&	- \sum d \left( \bar Z_{jr}^{-1} \dot{\bar Z}_{rk} Y_{ki}^{-1} \right)(w_l)w_l^* \otimes w_i \otimes \nabla_{p_j}s.
\end{align*}
So 
\begin{align*}
	\Delta_G s		=\,	& \tr(D''\circ G'\circ D')s  \\
				= 	& - 2i \sum \bar Z_{ik}^{-1} \dot{\bar Z}_{kl} Y_{lj}^{-1} C_{m}^{m,j} \nabla_{p_m}\nabla_{p_j} s \\
					& - \sum d(\bar Z_{ik}^{-1} \dot{\bar Z}_{kl} Y_{lj}^{-1}) (w_j) \nabla_{p_i}s \\
					& - k \sum \bar Z_{jk}^{-1}\dot{\bar Z}_{kl}^{-1} \bar Z_{lr}^{-1} Z_{rm}Y_{ms}^{-1} \omega(\bar w_s,p_j) s,
\end{align*}
here $\omega(\bar w_s,p_j)=-\bar Z_{sj}$.

Since we have that
\[
	Z=Z_{\infty}\frac{1}{t}+R(t), \quad \det Z_{\infty}\neq 0.
\]
where
\[
	t\cdot R(t) \to 0 \quad \text{ as } \quad t\to \infty
\]
and
\[
	t^2R'(t) \to 0 \quad \text{ as } \quad t\to \infty,
\]
we get that
\[
	Z^{-1}=t\cdot Z^{-1}_{\infty}(\id+tR(t)\cdot Z^{-1}_{\infty})^{-1}=t\cdot Z_{\infty}^{-1}+G(t)
\]
such that
\[
	\frac{1}{t}G(t)\to 0 \quad \text{ as } \quad t\to \infty.
\]
From this we see that
\[
	\bar Z^{-1}\dot{\bar Z}\bar Z^{-1}= (t\cdot \bar Z_{\infty}^{-1}+G(t)) \cdot (-\bar Z_{\infty}\cdot\frac{1}{t^2}+R'(t))\cdot(t\bar Z_{\infty}^{-1} + G(t)) = -\bar Z_{\infty}^{-1} + H(t)
\]	
where $H(t)\to 0$ as $t\to \infty$. Hence we have obtained the formula
\begin{equation}\label{lot}
	\lim_{t\to \infty} \Delta_{G}s= 2i \sum_{i,j} (\bar Z^{\infty})_{ij}^{-1} \nabla_{p_i} \nabla_{p_j} s
\end{equation}
Let now consider the first order term of $u(V)$.
\[
\sum_{i,j} 2 G^{i,j} \frac{\partial F}{\partial z_i}\nabla_{j}s =2 G\cdot \partial F \otimes \nabla^{1,0}s=-i \sum_{i,j} Y_{ik}^{-1}\dot{\bar Z}_{kl}Y_{lj}^{-1}dF(w_i) \otimes \nabla_{w_j}s
\]
So using $\nabla^{1,0}s=0$ we obtain
\begin{align*}
	2 G \partial F \otimes \nabla^{1,0}s 
		&= -2\sum_{i,j} Y_{ik}^{-1}\dot{\bar Z}_{kl}Y_{lj}^{-1}  dF(w_i) \otimes Y_{jr} \bar Z_{rs}^{-1}\nabla_{p_s}s \\
		&= -2 \sum_{i,j} Y_{ik}^{-1}\dot{\bar Z}_{kl} \bar Z_{lj}^{-1} dF(w_i)\otimes \nabla_{p_j}s.
\end{align*}

From this we get the following formula for 
\begin{align}\label{fibform}
\tilde{u}_{c,t}(s) = - \frac{1}{4k + 2n} \left( \right. & - 2i \sum \bar Z_{ik}^{-1} \dot{\bar Z}_{kl} Y_{lj}^{-1} C_{m}^{m,j} \nabla_{p_m}\nabla_{p_j} s \\
					& - \sum d(\bar Z_{ik}^{-1} \dot{\bar Z}_{kl} Y_{lj}^{-1}) (w_j) \nabla_{p_i}s \nonumber\\
					& - k \sum \bar Z_{jk}^{-1}\dot{\bar Z}_{kl}^{-1} \bar Z_{lr}^{-1} Z_{rm}Y_{ms}^{-1} \omega(\bar w_s,p_j) s\nonumber \\
					& -2 \sum_{i,j} Y_{ik}^{-1}\dot{\bar Z}_{kl} \bar Z_{lj}^{-1} dF(w_i)\otimes \nabla_{p_j}s\nonumber \\
					& + 4k \dot{F}_t s).\nonumber
\end{align}
where $\dot{F}_t$ refers to the derivative of $F_t$ with respect to the holomorphic part of $\tilde \sigma_t'$.
  
  \begin{claim}\label{C2}
    We have that
    \begin{itemize}
      \item The derivative along the directions of $F_{\tilde P, \bar \sigma} \cap \bar F_{\tilde P , \bar \sigma}$ of $F_t$ converges to zero.
      \item The derivative of $F_t$ with respect to the holomorphic part of $\tilde \sigma_t'$ goes to zero as $t$ goes to infinity.
      \item The function $F_t$ converges to zero, as $t$ goes to infinity.
    \end{itemize}
  \end{claim}

 \proof The claim follows directly from the equations which defines $F_t$ when combined with Theorem \ref{thm10}. \eproof

  From these two claims it follows immediately that $\tilde u_c(\tilde \sigma_t')$ has a limit, say $\tilde u_{c,\infty}$ as $t$ goes to infinity, and in fact
  \begin{align*}
		\tilde u_{c,\infty} = \Delta_{G_\infty}.
  \end{align*}
  \begin{claim}\label{C3}
    We have that the kernel of $\tilde u_{c,\infty}$ consists of sections that are convariant constant along $F_{\tilde P , \bar \sigma} \cap \bar F_{\tilde P , \bar \sigma}$.
  \end{claim}
  \begin{proof}
  We observe that $G_\infty$ induces a Hermitian structure on the leaves of $F_{\tilde P, \bar \sigma} \cap \bar F_{\tilde P , \bar \sigma} \cap TN'$ and that $\Delta_{G_\infty}$ is the corresponding Laplace--Beltrami operator associated to the restriction of $\nabla$ to the directions of $F_{\tilde P, \bar \sigma} \cap \bar F_{\tilde P, \bar \sigma} \cap TN'$. But then it follows immediately that the kernel of $\Delta_{G_\infty}$ are exactly the covariant constant sections of $\nabla$ along the directions of $F_{\tilde P , \bar \sigma} \cap \bar F_{\tilde P , \bar \sigma}$.
  \end{proof}
  Theorem \ref{prop1} now  follows directly from Claim \ref{C1}, \ref{C2} and \ref{C3} together with the above derived formulae.
\end{proof}

\begin{thm}
  \label{thm11}
  In the cases (1)---(3) above and for $\tilde P$ any admissible system of curves on $\tilde \Sigma$, there exists a limiting linear map
  \begin{align}
		\label{eq12}
		P_\infty (\tilde \sigma_0 , \tilde P) : H_{\sigma_0}^{(k)} \to H_{\tilde P,\bar \sigma_0}^{(k)}.
  \end{align}
\end{thm}
\begin{proof}
 Assume $E(t)$ is a solution to
 $$E'(t) = -P(t)E(t)$$
 where $P(t) = [\tilde u_{c,t}, \cdot]$ and  $E(t_0) = \Id$, where $t_0$ is some starttime. We further let $P_\infty=P(\infty) =  [\tilde u_{c,\infty}, \cdot].$
  Let now $Q(t) = e^{(t-t_0)P_\infty} E(t)$. Then
  \begin{align*}
	Q'(t) = E^{(t-t_0)P_\infty}(P_\infty-P(t))E(t),
\end{align*}
  so 
  \begin{align*}
		Q(t) =&\, \Id + \int_{t_0}^t e^{(s_0-t_0)P_\infty}(P_\infty-P(s_0))E(s_0)\, ds_0, \\
		E(t) =&\, e^{-(t-t_0)P_\infty} + \int_{t_0}^t e^{-(t-s_0)P_\infty}(P_\infty-P(s_0))E(s_0) \, ds_0 \\
		     =&\, e^{-(t-t_0)P_\infty} + \int_{t_0}^t e^{-(t-s_0)P_\infty}(P_\infty-P(s_0))e^{-(s_0-t_0)P}\, ds_0 \\
		      &\,+ \int_{t_0}^t e^{-(t-s_0)P_\infty}(P_\infty-P(s_0))\int_{t_0}^{s_0}e^{-(s_0-s_1)P_\infty}(P_\infty-P(s_1))E(s_1)\, ds_1 \, ds_0.
  \end{align*}
  Iterating this construction we arrive at the following formula
  \begin{align}
	  \label{formelstjerne}
	  E(t) =&\, \sum_{n=0}^\infty \int_{\Delta_n(t,t_0)}e^{-(t-s_0)P_\infty}(P_\infty-P(s_0))e^{-(s_0-s_1)P_\infty}(P_\infty-P(s_1)) \\
	    \\&\cdots (P_\infty-P(s_{n-1}))e^{-(s_{n-1}-t_0)P_\infty}\, ds_{n-1} \, \dots \, ds_0 \nonumber,
  \end{align}
  where
  \begin{align*}
		\Delta_n(t,t_0) = \{ (s_0,\dots,s_{n-1}) \in \bbR^n \mid t_0 \leq s_{n-1} \leq s_{n-2} \leq \dots \leq s_0 \leq t \}.
  \end{align*}
  We need to justify the convergence of the series \eqref{formelstjerne}. First we observe that
  \begin{align*}
		\vol(\Delta_n(t,t_0)) = \frac{(t-t_0)^n}{n!}.
  \end{align*}
  From the above we have that
  \begin{align*}
		\abs{P_\infty-P(t)} \leq ct^\alpha
  \end{align*}
  for all $t \in [t_0,\infty)$, where $\alpha < -1$. 
  This allows us to show that \eqref{formelstjerne} is absolutely summable. For large enough $t_0$ we will get that $\abs{e^{-tP_\infty}} = 1$ for all $t \leq t_0$. So then
  \begin{align*}
		\abs{\int_{\Delta_n(t,t_0)}& e^{-(t-s_0)P_\infty}(P_\infty-P(s_0)) \cdots (P_\infty-P(s_{n-1}))e^{-(s_{n-1}-t)P_\infty}ds_{n-1} \dots ds_0} \\
		&\leq c^n \abs{\int_{\Delta_n(t,t_0)} s_0^\alpha \cdots s_{n-1}^\alpha \, ds_{n-1}\dots ds_0}\\
		&= \frac{c^n}{n!}\left(-\frac{t_0^{\alpha+1}}{\alpha+1}+\frac{t^{\alpha+1}}{\alpha+1}\right)^n.
  \end{align*}
  Hence, we see that \eqref{formelstjerne} is summable and
  \begin{align*}
		\abs{E(t)} \leq e^{-\frac{ct_0^{\alpha+1}}{\alpha+1} + \frac{ct^{\alpha+1}}{\alpha+1}}.
  \end{align*}
  Note that the estimate converges to $e^{-\frac{ct_0^{\alpha+1}}{\alpha+1}}$ as $t \to \infty$.
  
  Let us now show that $E(t)$ is a Cauchy sequence as $t \to \infty$. Let $t_1 > t_2 > t_0$. Then
  \begin{align*}
	  \abs{E(t_1) - E(t_2)} \leq&\, \abs{\sum_{n=0}^\infty \int_{\Delta_n(T_2,t_2)} (e^{-t_1P_\infty}-e^{-t_2P_\infty})e^{s_0P}(O_\infty-P(s_0)) \cdots} \\
	  &\,+ \abs{\sum_{n=0}^\infty \int_{\Delta_n(t_1,t_0)-\Delta_n(t_2,t_0)}e^{-(t_1-s_0)P_\infty}(P-P(s_0)) \cdots} \\
	  \leq& \, \abs{e^{-t_1P_\infty}-e^{-t_2P_\infty}} e^{-\frac{ct_0^{\alpha+1}}{\alpha+1}}e^{\frac{ct_2^{\alpha+1}}{\alpha+1}} \\
	  &\,+ \abs{e^{\frac{ct_2^{\alpha+1}}{\alpha+1}}-e^{\frac{ct_1^{\alpha+1}}{\alpha+1}}} e^{-\frac{ct_0^{\alpha+1}}{\alpha+1}},
  \end{align*}
  which can be made arbitrary small provided $t_1$ and $t_2$ are large enough giving the Cauchy condition. Hence $E(\infty)$ exists. Moreover, by dividing by $\abs{t_1-t_2}$ and letting $t_2 \to t_1$. We see that $\abs{E'(t)}$ can be made arbitrarily small, provided $t$ is large enough, hence $E'(t) \to 0$ as $t \to \infty$. But then we get that
  \begin{align*}
		P_\infty E(\infty) = 0,
  \end{align*}
  proving $\Im E(\infty) \subseteq \ker P_\infty$. It is clear that $E(t)$ defined this satisfies the required equation. The theorem now follows from Claim \ref{C3}.
\end{proof}

Suppose we now have $s_P \in H^{(k)}_P$. Then we get an induced linear functional on $H^{(k)}_{\sigma_t}$ given by
\begin{align*}
	s_P(s) = \sum_{b \in B_P^{(k)}} \int_{x \in h_P^{-1}(b)} \langle s(x), s_P(x) \rangle \Vol_{\sigma_{t},b}(x),
\end{align*}
where $\Vol_{\sigma_{t},b}$ is the volume form on $h_P^{-1}(b)$ induced by the metric on $N$ associated to $\sigma_t$. Now let $s_{P,\sigma_t} \in H_{\sigma_t}^{(k)}$ be the state associated to this functional,
\begin{align*}
	(s, s_{P,\sigma_t}) = S_P(s),
\end{align*}
for all $s \in H_{\sigma_t}^{(k)}$.
\begin{prop}
  We have the following asymptotics in Teichmüller space:
  \begin{align*}
		\lim_{t \to \infty} P_\infty(\sigma_t,P) (s_{P,\sigma_t}) = s_P.
  \end{align*}
\end{prop}
\begin{proof} This Theorem follows by the same qrguments as in \cite{A11}, since the effect of degenerating the complex structure is after a local coordinate change equivalent to the large $k$ limit considered in \cite{A11}.
\end{proof}
\begin{cor}\label{Vor}
  In the cases (1)---(3) above and for $\tilde P$ any admissible system of curves on $\tilde \Sigma$, the map \eqref{eq12} is an isomorphism.
\end{cor}
Theorem~\ref{thm11} and this Corollary \ref{Vor} implies Theorem~\ref{thm3}.

\section{The four punctured sphere case}\label{4p}
Suppose $\Sigma$ is a $2$-sphere, and that $R$ consists of four points on $\Sigma$. Let $\tilde \Sigma = \Sigma - R$. Assume that we have a labeling $c : R \to [-2,2]$. Suppose we are given two transverse pair of pants decompositions $P_1$ and $P_2$ of $\tilde \Sigma$. Then $P_i = \{ \gamma_i \}$, where $\gamma_1$ and $\gamma_2$ are two transverse simple closed curves on $\tilde \Sigma$. We will use the notation $h_i = h_{P_i}$.

Choose an ordered subset $R'$ of $R$ of cardinality three. In this case we have the identity
\begin{align*}
	\tilde \calT \isom \bbC - \{0,1\}
\end{align*}
obtained as follows. For each $\tilde \sigma$, there is a unique $z \in \bbC - \{0,1\}$ and a unique biholomorphism from $(\Sigma_{\tilde \sigma},R)$ to $(\bbC P^1,\{0,1,\infty,z\})$ and which maps the ordered set $R'$ to the points $\{0,1,\infty\}$ on $\bbC P^1$.

In the following we determine the moduli space of flat
connections on a four
punctured sphere.

In stead of calculating the moduli spaces purely gauge theoretic we
will make heavy use of the identification of the moduli space of flat
connections on $\tilde\Sigma_{\tilde\s}$ with the character variety
$\calM(\tilde\Sigma_{\tilde\s}) = \Hom(\pi_1(\Sigma_{g,n}),\SU(2))/\SU(2)$. 

  There are many ways of calculating these moduli spaces. We could use
  the Morse theoretic approach as \cite{thaddeus}, or we could
  use pair of pants decomposition of $\tilde\Sigma_{\tilde\s}$ into two
  pair of pants glued along a circle, and calculate the fundamental
  group as an amalgamation of fundamental groups of two fundamental
  groups of a pair of pants. We will however calculate it by
  specifying specific curves, and use them to define coordinates in
  $\calM(\tilde\Sigma_{\tilde\s})$ by using trace. 
  
  Let $A,B,C,D$ be four curves on $\tilde\Sigma_{\tilde\s}$ each of which
  encircles a puncture. Then 
  \[
  \pi_1(\tilde\Sigma_{\tilde\s}) = \bracket{A,B,C,D
    \, | \, ABCD = 1}.
  \]We define seven coordinates on the moduli
  space, each for one of the trace of holonomies around the punctures $a = \Tr(\rho(A))$, $b = \Tr(\rho(B))$, $c = \Tr(\rho(C))$,
  $d = \Tr(\rho(D))$ and one for each of the belts dividing
  $\tilde\Sigma_{\tilde\s}$ into two pair of pants $x = \Tr(\rho(AB))$, $y =
  \Tr(\rho(BC))$ and a last for the diagonal $z = \Tr(\rho(AC))$,
  where $\rho$ is a $\SU(2)$-representation of
  $\pi_1(\tilde\Sigma_{\tilde\s})$. It can be shown (\cite{magnus}) that these
  functions satisfy the equation
  \begin{equation} \label{eq:trace-identity-2}
  x^2 + y^2 +z^2 + xyz = (ab+cd)x + (ad+bc)y + (ac+bd)z -
  (a^2+b^2+c^2+d^2 + abcd - 4).
  \end{equation}

  If the holonomies, $(\rho(A),\rho(B),\rho(C),\rho(D))$, around $A,B,C,D$ are fixed subject to $\rho(ABCD) =
  \Id$, the moduli space $N_{(\rho(A),\rho(B),\rho(C),\rho(D))}(\tilde\Sigma_{\tilde\s})$ is the
  zero-set of the polynomial \eqref{eq:trace-identity-2} in
  $[-2,2]^3$.
  For the permitted $(a,b,c,d) \in (-2,2)^4$ all moduli spaces
  are topologically spheres. In the six boundary cases 
  \[
  (a,b,c,d) \in \{
  (2,2,t,t), (2,t,t,2), (2,t,2,t), (t,t,2,2), (t,2,t,2),
  (t,2,2,t), t \in [-2,2]\},
  \] the moduli spaces are just points -- this
  corresponds to the case where two of the punctures has been filled
  in, and we consider the space of flat connections on a circle with
  specified holonomy $t \in [-2,2]$ --
  which is exactly a point.

\begin{remark}
  \label{rem:1}
  For a more detailed study of the moduli spaces mentioned in the
  above examples see e.g. \cite{Go1}.
\end{remark}

Let us now consider the moduli space of parabolic vector bundles on
$\tilde\Sigma_{\tilde\s}$. By the Mehta--Seshadri Theorem and the
calculations above this moduli space is generically a $2$-sphere.

Let $E \to \tilde\Sigma_{\tilde\s}$ be a stable parabolic vector
bundle of parabolic degree $0$ on $\tilde\Sigma_{\tilde\s}$ and let $L
\subset E$ be a proper subbundle. 

From above we have
\begin{align*}
  \pdeg L &= \deg L + \sum_{p \in R} w_1(p) \\
  &= \deg L + \sum_{\substack{p \in R \\ L_p = E_p^2}}s_p -
  \sum_{\substack{p \in R \\ L_p \neq E_p^2}}s_p \\
 &= \deg L + 2 \sum_{\substack{p \in R \\ L_p = E^2_p}}s_p - \sum_{p
   \in R} s_p
\end{align*}
For $E$ to be parabolically stable $\pdeg L < 0$ so we get the
following bound on the degree of $L$:
\[
\deg L = \pdeg L + \sum_{p \in R} s_p -2\sum_{\substack{p \in R \\ L_p
  = E_p^2}} \leq \sum_{p \in R}s_p.
\]
Since $\deg E = 0$ the Grothendieck classification of vector bundles
on $\bP^1$ give that $E \simeq \cO(k) \oplus \cO(-k)$ for an integer
$k \in \N$.

If $L= \cO(k)$ the restriction on degree gives $k \leq \sum_{p \in R}
s_p$. Now since there are four marked points and each of the $s_p$ are
less than $\frac{1}{2}$ we get that $k < 2$. Thus there are only two options 
\[
E \simeq \cO \oplus \cO \quad \text{or} \quad E \simeq \cO(1) \oplus \cO(-1).
\] 

Having analyzed this moduli space, we now turn to its quantization and the associated Hitchin connection. In particular, we will below identify 
the Hitchin connection explicitly with the TUY connection in the bundle of conformal blocks in this case of a four holed sphere. Hence let us first recall the sheaf of vacua construction from \cite{TUY}. 

Suppose $\lie g$ is a Lie algebra with a invariant inner product , which we will normalize such that
the longest root have length $\sqrt 2$. Let
$$B =  \bC -\{-1,0,1\}$$
and let $C= B \times \P^1$, which the canonical sections $s_i: B \ra C$, $i=1,2,3,4$ determined by 
$$s_1(\tau) = -1, \ s_2(\tau) = 0, \ s_3(\tau) = 1 \text{ and } s_4(\tau) = \tau,$$
for $\tau\in B$.
Let $\calF = (C,B, s_1,s_2,s_3,s_4)$ with the natural formal neighbourhoods induced from the canonical identification $\P^1 = \bC \cup \{\infty\}.$
Let 
\[
\hat {\lie g}(\calF) = \lie g \otimes _\bC H^0(C, \mathcal O_C(*\sum_{j=1}^N x_j))
\]
and recall from \cite{TUY} that the sheaf of conformal blocks over $B$ are given as follows
\[
	\V^\dagger_{\vec \lambda} (\calF )= \{\bra {\Psi}\in \mathcal O_B\otimes \cH^\dagger_{\vec 
	\lambda}\mid \bra {\Psi}\hat {\lie g} (\calF) =0\}
\]
where $\cH_{\lambda_i}$ is the heighest weight integrable $\hat {\lie g}$-module 
and
\[
	\cH^\dagger_{\lambda}= \cH^\dagger_{\lambda_1} \hat \otimes_\bC \dots 
	\hat \otimes_{\bC} \cH^\dagger_{\lambda_N}.
\]
As it is proved in \cite{TUY}, we get that the restriction map from 
$\cH_{\vec \lambda}$ to $\cH^{(0)}_{\vec \lambda}= V_\lambda$ induces an 
embedding of the sheaf of conformal block in genus $0$ into trivial $V_{\vec \lambda}^*$-bundle:
\[
 \V^\dagger_{\vec \lambda}(\calF) \hookrightarrow B \times ( V_{\vec \lambda}^*)^{\lie g}.
\]
Under this identification, the TUY-connection in the sheaf of conformal blocks gets identified with the KZ-connection in $B \times ( V_{\vec \lambda}^*)^{\lie g}$, which we now recall. Let $\Omega_{ij}$ is the quadratic Casimir acting in the $i$'th  and $j$'th 
factor.   Suppose that $(J_1, J_2, J_3)$ is an orthonormal basis of $\lie g$, 
then 
\[
 \Omega = \sum_{i=1}^3 J_i \otimes J_i
\]

So if $\rho_i : \SU(2)\to \Aut(V_{\lambda_i})$ and 
$\dot \rho_i : \lie g \to \End(V_{\lambda_i})$ are the representations of 
$\SU(2)$ and $\lie g$, and we  embed them into $\Aut(V_{\lambda_1}\otimes \dots \otimes V_{\lambda_4})$ and 
$\End(V_{\lambda_1}\otimes \dots \otimes V_{\lambda_4})$ in the usual way, then
\[
\Omega_{ij}= \dot \rho_i\otimes \dot \rho_j(\Omega).
\]
The KZ-connection is then given by
\[
	\nabla^\text{KZ}_{\frac \partial{\partial \tau}}= \nabla^t_{\frac 
	\partial{\partial \tau}} - \alpha(\frac\partial{\partial \tau}).
\]

where
\[
	\alpha(\frac\partial{\partial \tau}) = \frac {\Omega_{41}}{\tau} + 
	\frac{\Omega_{42}}{\tau-1}+ \frac{\Omega_{43}}{\tau+1}.
\]

We will now produce a geometric version of the KZ-connection.

The invariant inner product on the Lie algebra $\lie g$ induces a natural 
symplecitc structure on the coadjoint orbits. Let $\lie h\subseteq \lie g$ 
denote the Cartan subalgebra and denote by $X_\lambda$ the coadjoint 
orbit through $\lambda \in \lie h^*$.  If we use the right normalization of 
the inner product, we have that $X_\lambda$ is 
quantizable if and only if $\lambda$ is in the weight lattice.  Let $G=\SU(2)$ and assume that $\lambda$ is a dominant weight. We get a prequantum line bundle 
$\L_\lambda \to X_\lambda$, and the action of $\SU(2)$ lifts 
to this line bundle. Furthermore there exists a $\SU(2)$-invariant complex 
structure on $X_\lambda$.  It follows from the Bott--Borel--Weil Theorem that the representation of $\SU(2)$ on 
$H^0(X_\lambda, \L_\lambda)$  are the one determined by 
$\lambda$: 
\[
 V_\lambda \cong H^0(X_\lambda, \L_\lambda).
\]
The action of $\lie g$ on $V_\lambda$ can be described explicitly: we have an 
infinitesmal aciton of $\lie g$ on $X_\lambda$ given by
\[
	\lie g \to \mathcal X(X_\lambda) \quad \text{given by} \quad
        \xi \mapsto Z_\xi
\]
We then have that the action of $\lie g$ on $V_\lambda$ is described by
\[
	\xi(s) = \nabla_{x_\xi}s + 2\pi i \mu(\xi) s
\]
where $s\in H^0(X_\lambda, \L_\lambda)$ and $\mu(\xi)$ is the 
moment map evaluated on $\xi$. We remark that the action of $\lie g$ is given 
by  first order differential operators.  

Let us now consider the situation where we have four dominant weights
$\vec \lambda = (\lambda_1, \lambda_2, \lambda_3, \lambda_4)$, and consider 
the exterior tensor product
\[
	\L_{\vec \lambda} = p_1^*(\L_{\lambda_1}) \otimes  p_2^*(\L_{\lambda_2}) \otimes  p_3^*(\L_{\lambda_3}) \otimes  p_4^*(\L_{\lambda_4}) 
\]
which is a line bundle over 
\[
X=X_{\lambda_1}\times X_{\lambda_2}\times X_{\lambda_3}\times X_{\lambda_4}.
\]
  Thus we get a representation of $\SU(2)$  on
\[
 H^0(X,\L_{\vec \lambda}) \cong H^0(X_{\lambda_1}, \mathcal 
 L_{\lambda_1})\otimes H^0(X_{\lambda_2}, \mathcal 
 L_{\lambda_2})\otimes H^0(X_{\lambda_3}, \mathcal 
 L_{\lambda_3})\otimes H^0(X_{\lambda_4}, \mathcal 
 L_{\lambda_4}).
\]
We are interested in the invariant part
\[
	V_{\vec \lambda}^{G} = H^0(X,\L_{\bar \lambda})^{\SU(2)}.
\]
We can provide an alternative description of $V^G_{\vec \lambda}$ by 
applying the idea that quantization commutes with reduction:
Consider the moment map for the diagonal action
\[
	\mu : X_{\lambda_1}\times X_{\lambda_2} \times 
	X_{\lambda 3}\times X_{\lambda_4} \to \lie{g}^*
\]
given by
\[
\mu(\xi_1, \xi_2, \xi_3, \xi_4) = \sum_{i=1}^4 \xi_i.  
\]
Now we consider the symplectic reduction
\[
 \calM = \mu^{-1}(0)/\SU(2),
\]
which have an induced complex structure from $X$.  Furthermore there exists a 
unique line bundle $\L_\calM \to \calM$ s.t.  
\[
 p^*(\L_\calM) \cong \L_{\vec \lambda}|_{\mu^{-1}(0)}
\]
where $p: \mu^{-1}(0)\to \calM$ is the projection map.  

\begin{theorem}[Guillemin \& Sternberg] Quantization commutes with
  reduction, i.e.
	\[
		V^G \cong H^0(\calM,\L_\calM).
	\]
\end{theorem}
Now we consider the genus $0$ surface $\Sigma$ with $4$ marked points 
$x_1, \dots , x_4$.   We assume that we are provided with an identification 
$\Sigma \cong \P^1$, s.t.   $(x_1,x_2,x_3)$ are mapped to 
$(-1,0,1)$ and $x_4 $ to $\tau \in \P^1 - \{-1,0,1,\infty\}$. We 
assume that we have dominant weights $\lambda_1, \dots, \lambda_4$ attached to 
$x_1, \dots, x_4$.   The KZ-connection is defined as a connection in the 
trivial bundle
\[
	V_{\vec \lambda}^G = (V_{\lambda_1}\otimes \dots \otimes V_{\lambda_4})^G.  
\]
As stated above, the KZ-connection is described by the specific 1-form:
\[
	\nabla^\text{KZ}_{\frac \partial{\partial \tau}}= \nabla^t_{\frac 
	\partial{\partial \tau}} - \alpha(\frac\partial{\partial \tau})
\]
where
\[
	\alpha(\frac\partial{\partial \tau}) = \frac {\Omega_{41}}{\tau} + 
	\frac{\Omega_{42}}{\tau-1}+ \frac{\Omega_{43}}{\tau+1}
\]
From this we see that each of the operators $\Omega_{ij}$ become second order 
differential operators on $X$ acting on $\L_{\vec \lambda}$ such that 
they globally preserve $V_{\vec \lambda}^G$.   Let 
$u^\text{KZ}=u^{\text{KZ}}(\frac \partial {\partial \tau})$.   We now describe 
the resulting connection $\hat \nabla$  acting on the trivial 
$H^0(\calM,\L_\calM)$-bundle over $\P^1 - \{0,1,\infty\}$:
\[
\hat \nabla = \nabla^t-\hat u
\]
where $\hat u$ is a 1-form on $\P^1 - \{0,1,\infty\}$  with values 
in differential operators on $\calM$ acting on $\L_\calM$.   Explicitly we get 
a formula for $\hat u (\frac \partial{\partial z})$ by considering
\[
	X \supset \mu^{-1}(0) \to \calM
\]
and the splitting:
\[
	T_x\mu^{-1}(0) = T_x(Gx)\oplus (T_x(Gx))^\perp\cong T_x(Gx)\oplus p^*(T_x\calM).  
\]
of the tangent space of $\mu^{-1}(0)$ into a 3-dimensional and a 2-dimensional 
subspace.   Furthermore, we have that
\[
	T_x X = I(T_x(Gx))\oplus T_x\mu^{-1}(0)
\]
where $I$ is the complex structure on $X$.   On $G$-invariant section of 
$\L_{\vec \lambda}$ which are also holomorphic, i.e. $V^G_{\vec \lambda}$, 
we see that the derivatives in the direction of $T_x(Gx)$ and $I(T_x(Gx))$ 
vanishes, hence we can rewrite the action of $u^\text{KZ}$ as a second order 
differential operator which only differentiates in the direction of 
$(T(Gx))^\perp$.   Since we have $G$-invariance, we get this way an expression 
for $\hat u(\frac \partial{\partial z})$ as a second order differential operator.  

\begin{prop} The symbol of the second order differential operator
  $\hat{u}(\frac{\partial}{\partial z})$ is holomorphic, i.e.
	\[
		\sigma(\hat u(\frac\partial{\partial \tau} ))\in H^0(\calM,S^2(T))
	\]
\end{prop}
\begin{proof}
We observe that
\[
	S^2(T) \cong \mathcal O(4)
\]
under the identification of $\calM\cong \P^1$.   Next we observe that 
$u^{\text{KZ}}\in H^0(X,S^2(T))$ which then gives the stated result by reduction.
\end{proof}

We now compare this geometric version of the KZ-connection with the Hitchin connection.
Since $(\calM,\omega,I)$ is isomorphic to $\P^1$ as a complex manifold, we know 
there exists a smooth family of complex isomorphisms
\[
	\Phi_\tau :(\calM,I)\to \P^1
\]
varying smoothly with $\tau \in\T$. By 
comparing Chern-classes, we see that 
\[
	\Phi_\tau^*(\L_\calM) \cong \mathcal O(k_\calM)
\]
as holomorphic line bundls, for some $k_\calM \in \bZ$  independent of 
$\tau\in\T$.   From this we also get
\[
	G_\tau = \Phi^*_\tau\left(G\left(\frac \partial{\partial 
	\tau}\right)_\tau \right)\in H^0(\P^1, S^2(T \P^1))\cong H^0(\P^1,
	\mathcal O(4)).
\]
We observe that
\[
	S_0^2(H^0(\P^1,\mathcal O(2)))\cong H^0(\P^1,\mathcal O(4))
\]
as representations of $\SL(2,\bC)$.   Here we think of 
$S^2(H^0(\P^1,\mathcal O(2)))$ as quadratic forms on $H^0(\P^1, \mathcal O(2))$ 
and $S^2_0$ mean trace zero such.  

\begin{theorem}
\label{}
There exists
\[
	\Psi : \T \to \SL(2,\bC),
\]
such that if we define $\tilde \Phi_\tau = \Psi^{(\tau)}\circ
\Phi_\tau$ and let
\[
\tilde G_\tau = \tilde \Phi^*_\tau (G(\frac \partial{\partial \tau} )_\tau)\in 
H^0(\P^1,\mathcal O(4))
\]
then
\[
	\tilde G_{\tau}=\sigma(\hat \mu(\frac\partial{\partial \tau})).
\]
\end{theorem}
\begin{proof}
	We consider $S_0^2(H^0(\P^1,\mathcal O(2)))$ as a representation of 
	$\SL(2,\bC)$, where we think of $S_0^2(H^0(\P^1,\mathcal O(2)))$ as 
	quadratic forms on $H^0(\P^1,\mathcal O(2))=H^0(\P^1,T \P^1)$, hence 
	we consider elements of $S_0^2(H^0(\P^1,\mathcal O(2)))$ as symmetric 
	symmetric traceless $3\times 3$ complex matrices on which $\SL(2,\bC)$ 
	acts by conjugation.   We have that two symmetric traceless $3\times 3$ 
	complex matrices are conjugate if and only if they have the same eigenvalues.   
	An explicit computation shows that $\tilde G_\tau$ and 
	$\sigma(\hat u(\frac\partial{\partial \tau} ))$ has the same 
	eigenvalues, hence we can find the required map $\Psi$.  
\end{proof}

Since $\tilde \Phi$ is such that the two symbols of the two second order differential operators defining the Hitchin connection and the geometric KZ-connection have been aligned, it follows from the form the Hitchin connection has, in order to preserve the subbundle of holomorphic sections that $\tilde \Phi$ must take the Hitchin connection to the KZ-connection. We further see that the Bohr-Sommerfeld decomposition corresponding to the limiting real polarizations, when $\tau$ approaches $-1$ and $1$, corresponds to the factorization decomposition for the covariant constant sections of the sheaf of vacua constructed in \cite{TUY}.
\begin{thm}
  \label{thm12}
  If $P_1$ and $P_2$ are pair of pants decompositions related by an elementary flip on a four-punctured sphere, then $[\cdot,\cdot]_{P_1,\sigma_0}$ and $[\cdot,\cdot]_{P_1,\sigma_0}$ are projectively equivalent.
\end{thm}

\proof
The projective equivalence is obtained by the tensor product of the parallel transport discussed above on the four-punctured sphere in question with the identity on the complementary part in the factorization. The fact that this is a projective equivalence follows from the above arguments identifying the parallel transport of the Hitchin connection with the KZ-connection, which by the results of \cite{AU1,AU2,AU3,AU4} know is an isometry, since the corresponding flip transformation in the Reshetikhin-Turaev TQFT is an isometry.

\eproof

\section{The once punctured genus one case}\label{elliptic}
%Let $\Sigma$ be a genus $1$ curve with one marked point $P$. Now we pick% a conjugacy class $c_0 \in [-2,2]$ and let $N$be the moduli space of fl%at connections on $\Sigma - P$ with holonomy $c_0$ around $P$.

%\begin{prop}
%  If $c_0 \not= 2$, then $N$ is smooth and it is diffeomorphic to $S^2$.
%\end{prop}
%\begin{proof}
%  We have of course that
%  \begin{align*}
%		N = \{(A,B) \in SU(2)^{\times 2} \mid [A,B] \in c_0 \}  /% SU(2),
%  \end{align*}
%  from which it is clear that $N$ contains no reducible connections, hen%ce $N$ is smooth. Now the holonomy f unction for any non-trivial non-pun%cture parallel curve induces a Morse function on $N$ with one maximum an%d one minimum and no other critical points, hence $N$ must be diffeomorp%hic to $S^2$.
%\end{proof}

Consider the specific case of a torus with a single puncture,
$\tilde\Sigma_{\tilde\s}$ (in the notation above). Let $N_{c_0}$ be the
moduli space of flat connections on $\tilde\Sigma_{\tilde\s}$ with $c_0
\in [-2,2]$ the holonomy around the puncture. The generators of the fundamental group are the curves $a,b,c$ being the longitude, meridian and a small curve around the puncture. The fundamental group of $\tilde\Sigma_{\tilde\s}$ is $\pi_1(\tilde\Sigma_{\tilde\s}) = \bracket{a,b,c \, | \, aba^{-1}b^{-1} = c}$. 

Let $\rho: \pi_1(\tilde\Sigma_{\tilde\s}) \to \SU(2)$ be a
$\SU(2)$-representation of $\pi_1(\tilde\Sigma_{\tilde\s})$. Define $A =
\rho(a)$, $B = \rho(b)$ and $C = \rho(c)$. We describe the moduli space by determining each of the fibers of the trace $\Tr: N \to [-2,2]$.

The case where $C$ corresponds to minus the identity (i.e. $\Tr(C)
= -2$) is the same as removing the puncture. Now since $a,b$ commute in
$\pi_1(\tilde\Sigma_{\tilde\s})$ we have $AB = BA$. Every element of
$\SU(2)$ can be diagonalized, so as $\SU(2)$ acts on the
representation variety by diagonal conjugation we assume $A$ to
be diagonal. Assume also that $A$ has distinct eigenvalues. Then
the only element $B$ that commutes with $A$ are diagonal
matrices. Hence $A$ and $B$ can be simultaneously
diagonalised to be elements of $S^1$. We can however still conjugate
$A$ and $B$ by elements of the Weyl group and still stay
within $S^1 \subset \SU(2)$ (this amounts to changing the order of the
eigenvalues), so $N_1(\tilde\Sigma_{\tilde\s}) = S^1 \times S^1 / \Z_2$. In
the case of $A$ or $B$ not having two distinct eigenvalues the above
description is still valid; generally however these non-generic cases
correspond to singular points of the moduli space.

Let $a,b,c$ be curves as above. The trace provides coordinates on the
moduli space, so let $\rho$ be a $\SU(2)$-representation of the
fundamental group, and define $x = \Tr(\rho(a))$, $y = \Tr(\rho(b))$
and $z = \Tr(\rho(ab))$. The moduli space is a subset of $[-2,2]^3$
carved out by the relation from the presentation of the fundamental group. Now fix the holonomy around $c$ to be $C \in
\SU(2)$. By the relation $ABA^{-1}B^{-1} = C$, and it is a simple
check that the following identity is satisfied for any $A,B \in \SU(2)$:
\begin{equation} \label{eq:trace-identity}
\Tr(ABA^{-1}B^{-1}) = \Tr(A)^2 + \Tr(B)^2 + \Tr(AB)^2 - \Tr(A)\Tr(B)\Tr(AB)-2.
\end{equation}
In other words the moduli space with fixed holonomy around $c$ is
\[
N_{c_0}(\tilde\Sigma_{\tilde\s}) = \{(x,y,z) \in [-2,2]^3 \, | \, x^2 + y^2 + z^2
- xyz - 2 = c_0\},
\]
which is topologically a sphere, for all values of $c_0 \in (-2,2]$.

We expect that we can find an argument completely parallel to the one given above in the genus zero case, since the moduli space is again a sphere. However we do not strictly need this, since by \cite{AU3}, we know that the genus zero part of a modular functor determines $S$-matrix, which is the need equivalence in this case. By the result of the previous section, we know that the quantization of the moduli spaces of does indeed give a modular functor which in genus zero is isomorphic to the one constructed in \cite{AU2} for the Lie algebra of $SU(2)$. Hence we have the following theorem.

\begin{thm}
  \label{thm13}
  If $P_1$  and $P_2$ are pair of pants decompositions related by an elementary flip on a once punctured torus, then $[\cdot,\cdot]_{P_1,\sigma_0}$ and $[\cdot,\cdot]_{P_2,\sigma_0}$ are projectively equivalent.
\end{thm}

\section{Well-definedness of the projective Hermitian structure}
\label{sect7}
We recall the setting from the introduction, where $\Sigma$ is a closed oriented surface of genus $g > 1$ and $P$ is a pair of pants decomposition of $\Sigma$. Recalling the map \eqref{eq1}, we define the representative $[\cdot,\cdot]_P^{(k)}$ of $(\cdot,\cdot)^{(k)}$ determined by $P$ by the formula
\begin{align*}
	[s_1,s_2]^{(k)}_{P,\sigma_0} = (P_\infty(\sigma_0,P)(s_1),P_\infty(\sigma_0,P)(s_2))_P^{(k)},
\end{align*}
for all $s_1,s_2 \in H^{(k)}_{\sigma_0}$.
\begin{thm}
  \label{thm14}
  The Hermitian structure $[\cdot,\cdot]^{(k)}_P$ is projectively preserved by the Hitchin connection.
\end{thm}
\begin{proof}
  We consider two arbitrary complex structures $\sigma_1$ and $\sigma_2$. Parallel transport along any curve from $\sigma_1$ to $\sigma_2$ is invariant up to scale under perturbation of the curve, hence the curve can be deformed to the canonical curve from $\sigma_1$ to $P$ and composed with the reverse of the canonical curve from $\sigma_2$ to $P$ without changing the projective class of the parallel transport. But by the definition of $[\cdot,\cdot]_P^{(k)}$, the result now follows.
\end{proof}
\begin{thm}
  \label{thm15}
  For any two pair of pants decompositions $P_1$ and $P_2$ on $\Sigma$, any complex structure on $\sigma_0$ on $\Sigma$ and any level $k$, we have that $[\cdot,\cdot]^{(k)}_{P_1}$ and $[\cdot,\cdot]^{(k)}_{P_2}$ induce the same projective unitary structure on $H^{(k)}$.
\end{thm}
\begin{proof}
  This is an immediate consequence of Theorem~\ref{thm12} and \ref{thm13}.
\end{proof}

\section{Comparison with the $L^2$ Hermitian structure}
\label{sect8}
In order to analyze the large $k$ asymptotics of the Hermitian structures $[\cdot,\cdot]^{(k)}_P$, we return to the map
\begin{align*}
	h_P : M \to [-2,2]^{3g-3}.
\end{align*}
We now compose this map with the inverse of the map from $[0,1]$ to $[-2,2]$ given by $t \mapsto 2\cos(t\pi$ to obtain the map
\begin{align*}
	\tilde{h}_P : M \to [0,1]^{3g-3}.
\end{align*}
Define
\begin{align*}
	\calD = \{z \in [0,1]^{3g-3} \mid \forall v \in V_{\Gamma_P}, \,\ z(v) \in T \},
\end{align*}
where
\begin{align*}
	T = \{z_1,z_2,z_3 \in [0,1] \mid \abs{z_1-z_2} \leq z_3 \leq z_1+z_2,z_2+z_2+z_3 \leq 2\}.
\end{align*}
In \cite{JW}, the following Propositions is established.
\begin{prop}
  We have that
  \begin{align*}
		\tilde{h}_P(M) = \calD.
  \end{align*}
\end{prop}

Now we define the function $G_P^{(k)} : M \to \bbC$ by
\begin{align*}
	G_P^{(k)} = H^{(k)} \circ \tilde{h}_P
\end{align*}
where $H^{(k)} : \D \ra \bR_+$ is a function we determine below.
In order to extend $G^{(k)}_P$ from being defined just at the pair of pants $P$ to being defined also for points in the interior of $\calT$, we introduce the following bundles. Let $\calC = \calT \times C^\infty(M)$. We define a subbundle $\calN^{(k)}$ of $\calC$ whose fibers are
\begin{align*}
	\calN_\sigma^{(k)} = \ker(T^{(k)}_\sigma)
\end{align*}
for $\sigma \in \calT$ and where
\begin{align}
	\label{eq14}
	T_\sigma^{(k)} : C^\infty(M) \to \End(H_\sigma^{(k)})
\end{align}
is the Toeplitz map given by
\begin{align*}
	f \mapsto T^{(k)}_{\sigma,f} = \pi_\sigma^{(k)} \circ M_f.
\end{align*}
Here $\pi_\sigma^{(k)}$ is the orthorgonal projection onto $H^0(M_\sigma,L^k)$ and $M_f : C^\infty(M) \to C^\infty(M)$ is the multiplication operator
\begin{align*}
	M_f(s) = fs
\end{align*}
defined for all $s \in C^\infty(M,L^k)$ and any $f \in C^\infty(M)$. Now we introduce the quotient subbundle
\begin{align*}
	\calC^{(k)} = \calC/\calN^{(k)}.
\end{align*}
Since the Toeplitz map is surjective \cite{BMS}, we know that $\calC^{(k)}$ is a vector bundle over $\calT$ which is isomorphic to $\End(H^{(k)})$ over $\calT$ via the linear bundle isomorphism $T^{(k)}$.

Let $C^\infty_\bbR(M)$ be the subspace of $C^\infty(M)$ consisting of the real valued functions. We let $\calN_\bbR \subset \calN^{(k)}$ and $\calC_\bbR^{(k)} \subset \calC^{(k)}$ consist of the corresponding subbundles of real valued functions
\begin{thm}
  Hermitian structures on $H^{(k)}$ are under the isomorphism $T^{(k)}$ in one to one correspondence between smooth sections of $C_\bbR^{(k)}$ over $\calT$. The subset of Hermitian structures which are projectively preserved by the Hitchin connection are in one to one correspondence with sections of $C_\bbR^{(k)}$ which are preserved projectively by a certain flat connection $D^{(k)}$ acting on smooth sections of $C_\bbR^{(k)}$.
\end{thm}
\begin{proof}
Suppose $G : \T \ra C^\infty_\bbR(M)$. Then we get a Hermitian structure on $H^{(k)}$ by the fomula
$$ (s_1,s_2)_{G,\sigma} = \int_M \langle s_1,s_2\rangle G_\sigma \frac{\omega^n}{n!}.$$
We see that $(\cdot,\cdot)_G$ is projectively preserved if an only if 
$$ \pi^{(k)}_\sigma V[G] + \pi^{(k)}_\sigma G u(V) + \pi^{(k)}_\sigma  u(V)^* G = c_\sigma \Id$$
for all vector fields $V$ on $\T$ and where $c$ is some $c \in C^{\infty}(\T)$.
Now a simply rewrite of this formula using the techniques from \cite{A6} gives an explicit formula for $D^{(k)}$.
\end{proof}

Let $P$ be a pair of pants. We now introduce a top form on the fibers of $h_P$ as follows. We use the symplectic form to provide an isomorphism, at a generic point, between the top exterior power of the cotangent bundle along the fibers and the top exterior power of the tanget space to $\calD$ at the image of the point under $\tilde{h}_P$. Over $\calD \subset [0,1]^{3g-3}$ we have a canonical section of the top exterior power of the tangent bundle, which we use to induce a volume form on the fibers. We denote the fiberwise volume form $\Omega_{P,b}$. For each of the Bohr--Sommerfeld fiber $b$ of $\tilde{h}_P$, we introduce a projection operator
\begin{align*}
	\pi_{P,b}^{(k)} : C^\infty(\tilde h_P^{-1}(b),\calL^k) \to H_{P,b}^{(k)},
\end{align*}
where $H_{P,b}^{(k)}$ is the subset of $H_P^{(k)}$ consisting of covariant constant sections with support on $\tilde{h}^{-1}_P(b)$ associated to the inner product on $C^\infty(\tilde{h}^{-1}_P(b),\calL^k)$ given by
\begin{align*}
	(s_1,s_2)^{(k)}_{P,b} = \int_{\tilde{h}_P^{-1}(b)} \langle s_1, s_2 \rangle \Omega_{P,b}.
\end{align*}
We can now define
\begin{align*}
	T_P^{(k)} : C^\infty(M) \to \End(H_P^{(k)})
\end{align*}
as the composite
\begin{align*}
	T_P^{(k)} = \bigoplus_{b \in B_k(P)} \pi_{P,b}^{(k)} \circ M_{f|_{\tilde{h}^{-1}_P(b)}}.
\end{align*}
Now we let
\begin{align*}
	\calN_P = \ker T_P^{(k)}
\end{align*}
and
\begin{align*}
	\calC_{\bbR,P}^{(k)}= C^\infty(M) / \calN_{\bbR,P},
\end{align*}
where $\calN_{\bbR,P}$ is the real part of $\calN_P$.

Let
\begin{align*}
	P_t^\calC(\sigma_0,P) : \calC_{\bbR,\sigma_0}^{(k)} \to \calC_{\bbR,\sigma_t}^{(k)}
\end{align*}
be the parallel transport with respect to the connection $D^{(k)}$. Now we introduce a sub-bundle $\calC_{\bbR,D}^{(k)}$ of $\calC_{\bbR}^{(k)}$ whose fiber over a $\sigma\in \T$ consists of those equivalence classes of functions which at under $T^{(k)}_\sigma$ is taken to operators which acts diagonally with respect to the direct sum decomposition 
$$ H^{(k)}_\sigma = \bigoplus_{b\in B^{(k)}_P} P_\infty(\sigma,P)^{-1}(H^{(k)}_{P,b}).$$
\begin{thm}
  The operators $P_t^\calC(\sigma_0,P)|_{\calC_{\bbR,D,\sigma_0}^{(k)}}$ has a well-defined limit
  \begin{align*}
	  P_\infty^\calC(\sigma_0,P)|_{\calC_{\bbR,D,\sigma_0}^{(k)}} : \calC_{\bbR , D, \sigma_0}^{(k)} \to \calC^{(k)}_{\bbR,P}
  \end{align*}
  which is an isomorphism.
\end{thm}

This theorem follows immediately from the above results.

Using the function $G_P^{(k)}$, we can extend $(\cdot,\cdot)^{(k)}_P$ to a Hermitian structure on $C^\infty(M,\calL^k)$ via the formula
\begin{align*}
	(s_1,s_2)_P^{(k)} = \int_M \langle s_1, s_2 \rangle G_P^{(k)}\frac{\omega^m}{m!},
\end{align*}
for $s_1,s_2 \in C^\infty(M,\calL^k)$.

We define the representative $[\cdot,\cdot]_P^{(k)}$ of $(\cdot,\cdot)^{(k)}$ determined by $P$ by the formlua
\begin{align*}
	[s_1,s_2]_{P,\sigma_0}^{(k)} = (P_\infty(\sigma_0,P)(s_1),P_\infty(\sigma_0,P)(s_2))_P^{(k)},
\end{align*}
for all $s_1,s_2 \in H_{\sigma_0}^{(k)}$.
We observe that we can embed $\calC_{\bbR}^{(k)}$ into the trivial bundle $C^\infty(M)\times \T$, where it maps onto the functions which are orthogonal to the subbundle $\N^{(k)}$. This allows us to define $G^{(k)} : \T \ra C^\infty(M)$ which is orthogonal to $\N^{(k)}$ and such that
it projects to a covariant constant section of $\calC_{\bbR,D}^{(k)}$ and limits to $G_P$ at $P$, where we have chosen $H^{(k)}=1$.
From the construction, $G^{(k)}$ seems to depend $P$, but since by construction we have the following theorem, it actually does not.
\begin{thm}\label{thm18}
We have that
$$[s_1,s_2]_{P,\sigma}^{(k)} = \int_M \langle s_1, s_2 \rangle G_\sigma^{(k)}\frac{\omega^m}{m!},
$$
for all $s_1,s_2\in H^{(k)}_{\sigma}$ and all $\sigma\in \T$.
\end{thm}

Now we simply just need to observe that there is an asymptotic expansion of $G^{(k)}_\sigma$ in terms of $1/k$ by its very construction and by Claim \ref{C2} we further the theorem below and Theorem \ref{thm6}.

\begin{thm}
  \label{thm19}
  The Hermitian structures $[\cdot,\cdot]_P^{(k)}$ on $H^{(k)}$ are uniformly equivalent to $(\cdot,\cdot)_{L^2}^{(k)}$, i.e. for each $\sigma_0 \in \calT$, there is a constant $c_{P,\sigma_0}$ such that
  \begin{align*}
	c^{-1}_{P,\sigma_0} \abs{s}_{L^1} \leq \abs{s}_{P,\sigma_0} \leq c_{P,\sigma_0}\abs{s}_{L^2}
\end{align*}
for all $s \in H_{\sigma_0}^{(k)}$.
\end{thm}


\begin{thebibliography}{0}


\bibitem[A1]{A1} J.E. Andersen, "Jones-Witten theory and the Thurston boundary of Teichm\"{u}ller space", University of Oxford D. Phil thesis (1992), 129pp.


\bibitem[A2]{A2} J.E. Andersen, "Geometric Quantization of Symplectic Manifolds with
respect to reducible non-negative polarizations". Commun. in Math.
Phys. {\bf 183}, (1997), 401--421.

\bibitem[A3]{A3} J.E. Andersen, "New polarizations on the moduli space and the Thurston
compactification of Teichmuller space", International Journal of
Mathematics, {\bf 9}, No.1 (1998), 1--45.


\bibitem[AM]{AM} J.E. Andersen \& G. Masbaum, "Involutions on moduli spaces and
refinements of the Verlinde formula". Math Annalen {\bf 314}
(1999), 291--326.


\bibitem[A4]{A4} J.E. Andersen, "The asymptotic expansion conjecture", section 7.2 of "Problems on invariants of knots and 3-manifolds" Edited by T. Ohtsuki, in "Invariants of knots and 3-manifolds (Kyoto 2001)", Editors: Tomotada Ohtsuki, Toshitake Kohno, Thang Le, Jun Murakami, Justin Roberts and Vladimir Turaevin, Geometry \& Topology Monographs, {\bf 4}, (2002), 747--754.

\bibitem[A5]{A5} J.E Andersen, "Deformation quantization
and geometric quantization of abelian moduli spaces.",
 Comm. of Math. Phys. {\bf 255}  (2005), 727--745.

\bibitem[A6]{A6} J.E. Andersen, "Asymptotic faithfulness of the quantum $SU(n)$
representations of the mapping class groups". Annals of
Mathematics, {\bf 163} (2006), 347--368.


\bibitem[AH]{AH} J.E Andersen \& S.K. Hansen, "Asymptotics of the quantum
invariants of surgeries on the figure 8 knot", Journal of Knot
theory and its Ramifications, {\bf 15} (2006), 1--69.

\bibitem[AMU]{AMU} J.E. Andersen, G. Masbaum \& K. Ueno, "Topological quantum field theory
and the Nielsen-Thurston classification of $M(0,4)$", Math. Proc.
Cambridge Philos. Soc. 141 (2006), no. 3, 477--488.

\bibitem[AU1]{AU1} J.E. Andersen \& K. Ueno, "Geometric construction of modular functors
from conformal field theory", Journal of Knot theory and its
Ramifications. {\bf 16} 2 (2007), 127--202.

\bibitem[AU2]{AU2} J.E. Andersen \& K. Ueno, "Abelian Conformal Field theories and
Determinant Bundles", International Journal of Mathematics. {\bf
18}, (2007) 919--993.

\bibitem[A7]{A7} J.E. Andersen, "The Nielsen-Thurston classification of mapping classes is
determined by TQFT",  J. Math. Kyoto Univ. {\bf 48} 2 (2008), 323--338. 


\bibitem[A8]{A8} J.E. Andersen, "Toeplitz Operators and Hitchin's projectively flat
connection", in {\it The many facets of geometry: A tribute to
Nigel Hitchin}, 177--209, Oxford Univ. Press, Oxford, 2010.


\bibitem[AG]{AG} J.E. Andersen \& N.L. Gammelgaard, "Hitchin' s Projectively Flat Connection, Toeplitz Operators and the Asymptotic Expansion of TQFT Curve Operators",  
Grassmannians, Moduli Spaces and Vector Bundles, 1--24, Clay Math. Proc., 14, Amer. Math. Soc., Providence, RI, 2011.


\bibitem[AB]{AB} J.E. Andersen \& J. Blaavand, "Asymptotics of Toeplitz operators and applications in TQFT", Traveaux Math\'{e}matiques, {\bf 19} (2011), 167--201.


\bibitem[AU3]{AU3} J.E. Andersen \& K. Ueno, "Modular functors are determined by
their genus zero data", Quantum Topology {\bf 3} 3/4 (2012) 255--291.



\bibitem[A9]{A9} J.E. Andersen, "Hitchin's connection, Toeplitz operators and symmetry invariant
deformation quantization", Quantum Topology {\bf 3} 3/4 (2012) 293--325.


\bibitem[AGL]{AGL} J.E. Andersen, N.L. Gammelgaard \& M.R. Lauridsen,
"Hitchin's Connection in Metaplectic Quantization", Quantum Topology {\bf 3} 3/4 (2012) 327--357.


\bibitem[AH]{AH} J.E. Andersen \& B. Himpel, "The Witten-Reshetikhin-Turaev invariant of finite order mapping tori II", Quantum Topology {\bf 3} 3/4 (2012) 377--421.

\bibitem[A10]{A10} J.E. Andersen, "The Witten-Reshetikhin-Turaev invariant of finite order mapping tori I", 
Journal f\"{u}r Reine und Angewandte Mathematik. Published online 24/4 2012:  DOI: 10.1515/crelle-2012-0033. Available at \\
http://www.degruyter.com/view/j/crelle.ahead-of-print/crelle-2012-0033/crelle-2012-0033.xml?format=INT

\bibitem[A11]{A11} J.E. Andersen, "Mapping Class Groups do not have Kazhdan's Property (T)",
arXiv:0706.2184, pp. 21.

\bibitem[AU4]{AU4} J.E. Andersen \& K. Ueno, "Construction of the Reshetikhin-Turaev TQFT from conformal field theory", arXiv:1110.5027, pp. 39.



\bibitem[At]{At} M. Atiyah, \emph{The Jones-Witten invariants of
    knots}. S\'{e}minaire Bourbaki, Vol. 1989/90. Ast\'{e}risque No.
  {\bf 189-190} (1990), Exp. No. 715, 7--16.

\bibitem[AB]{AB} M. Atiyah \& R. Bott, \emph{The Yang-Mills equations
    over Riemann surfaces}.  Phil. Trans. R. Soc. Lond., Vol.  {\bf
    A308} (1982) 523--615.


\bibitem[ADW]{ADW} S.~Axelrod, S.~Della~Pietra, E.~Witten,
  \emph{Geometric quantization of Chern Simons gauge theory},
  J.Diff.Geom. {\bf 33} (1991) 787--902.









\bibitem[BK]{BK} B. Bakalov and A. Kirillov, \emph{Lectures on tensor
    categories and modular functors}, AMS University Lecture Series,
  {\bf 21} (2000).

\bibitem[BHV]{BHV} B. Bekka, P. de la Harpe \& A. Valette,
  \emph{Kazhdan's Proporty (T)}, In Press, Cambridge University Press
  (2007).

\bibitem[Besse]{Besse} A. L. Besse, \emph{Einstein Manifolds},
  Springer-Verlag, Berlin (1987).



\bibitem[B1]{B1}C. Blanchet, \emph{Hecke algebras, modular categories
    and $3$-manifolds quantum invariants}, Topology {\bf 39} (2000),
  no. 1, 193--223.

\bibitem[BHMV1]{BHMV1} C. Blanchet, N. Habegger, G. Masbaum \&
  P. Vogel, \emph{Three-manifold invariants derived from the Kauffman
    Bracket}.  Topology {\bf 31} (1992), 685--699.


\bibitem[BHMV2]{BHMV2} C. Blanchet, N. Habegger, G. Masbaum \&
  P. Vogel, \emph{Topological Quantum Field Theories derived from the
    Kauffman bracket}. Topology {\bf 34} (1995), 883--927.


\bibitem[BC]{BC} S. Bleiler \& A. Casson, \emph{Automorphisms of
    sufaces after Nielsen and Thurston}, Cambridge University Press,
  1988.

\bibitem[BMS]{BMS} M. Bordeman, E. Meinrenken \& M.  Schlichenmaier,
  \emph{Toeplitz quantization of K{\"a}hler manifolds and $gl(N), N
    \ra \infty$ limit}, Comm. Math. Phys. {\bf 165} (1994), 281--296.

\bibitem[BdMG]{BdMG} L. Boutet de Monvel \& V. Guillemin, \emph{The
    spectral theory of Toeplitz operators}, Annals of Math. Studies
  {\bf 99}, Princeton University Press, Princeton.

\bibitem[BdMS]{BdMS} L. Boutet de Monvel \& J. Sj\"{o}strand,
  \emph{Sur la singularit\'{e} des noyaux de Bergmann et de
    Szeg\"{o}}, Asterique {\bf 34-35} (1976), 123--164.


\bibitem[DN]{DN} J.-M. Drezet \& M.S. Narasimhan, \emph{Groupe de
    Picard des vari\'{e}t\'{e}s de modules de fibr\'{e}s semi-stables
    sur les courbes alg\'{e}briques}, Invent. math. {\bf 97} (1989)
  53--94.



\bibitem[Fal]{Fal} G.~Faltings, \emph{Stable G-bundles and projective
    connections}, J.Alg.Geom. {\bf 2} (1993) 507--568.

\bibitem[FLP]{FLP} A. Fathi, F. Laudenbach \& V. Po\'{e}naru,
  \emph{Travaux de Thurston sur les surfaces}, Ast\'{e}risque {\bf
    66--67} (1991/1979).

\bibitem[Fi]{Fi} M. Finkelberg, \emph{An equivalence of fusion
    categories}, Geom. Funct. Anal. {\bf 6} (1996), 249--267.

\bibitem[Fr]{Fr} D.S. Freed, \emph{Classical Chern-Simons Theory, Part
    1}, Adv. Math. {\bf 113} (1995), 237--303.


\bibitem[FWW]{FWW} M. H. Freedman, K. Walker \& Z. Wang, \emph{Quantum
    $\SU(2)$ faithfully detects mapping class groups modulo center}.
  Geom. Topol. {\bf 6} (2002), 523--539

\bibitem[FR1]{FR1} V. V. Fock \& A. A Rosly, \emph{Flat connections
    and polyubles}. Teoret.  Mat. Fiz. {\bf 95} (1993), no. 2,
  228--238; translation in Theoret. and Math. Phys. {\bf 95} (1993),
  no. 2, 526--534

\bibitem[FR2]{FR2} V. V. Fock \& A. A Rosly, \emph{Moduli space of
    flat connections as a Poisson manifold}.  Advances in quantum
  field theory and statistical mechanics: 2nd Italian-Russian
  collaboration (Como, 1996). Internat. J. Modern Phys. B {\bf 11}
  (1997), no. 26-27, 3195--3206.

\bibitem[vGdJ]{vGdJ} B. Van Geemen \& A. J. De Jong, \emph{On
    Hitchin's connection}, J. of Amer. Math. Soc., {\bf 11} (1998),
  189--228.


\bibitem[Go1]{Go1} W. M. Goldman, \emph{Ergodic theory on moduli
    spaces}, Ann. of Math. (2) 146 (1997), no. 3, 475--507.


\bibitem[Go2]{Go2} W. M. Goldman, \emph{Invariant functions on Lie
    groups and Hamiltonian flows of surface group representations},
  Invent. Math. 85 (1986), no. 2, 263--302.

\bibitem[GS]{GS} V. Guillemin and S. Sternberg, "Geometric Asymptotics", Mathematical Surveys, {\bf 14}, American Mathematical Society, Providence, Rhode Island, (1977).

\bibitem[GR]{GR} S. Gutt \& J. Rawnsley, \emph{Equivalence of star
    products on a symplectic manifold}, J. of Geom. Phys., {\bf 29}
  (1999), 347--392.


\bibitem[H]{H} N.~Hitchin, \emph{Flat connections and geometric
    quantization}, Comm.Math.Phys., {\bf 131} (1990) 347--380.


\bibitem[JW]{JW} L. Jeffrey \& J. Weitsman, \emph{Bohr-Sommerfeld orbits in the moduli space of flat connections and the Verlinde dimension formula}. Comm. Math. Phys. {\bf 150} (1992) 593 -- 630.

\bibitem[KS]{KS} A. V. Karabegov \& M.  Schlichenmaier,
  \emph{Identification of Berezin-Toeplitz deformation quantization},
  J. Reine Angew. Math. {\bf 540} (2001), 49--76.

\bibitem[Kar]{Kar} A. V. Karabegov, \emph{Deformation Quantization
    with Separation of Variables on a K\"ahler Manifold},
  Comm. Math. Phys. {\bf 180} (1996) (3), 745---755.

\bibitem[Kac]{Kac} V. G. Kac, \emph{Infinite dimensional Lie
    algebras}, Third Edition, Cambridge University Press, (1995).

\bibitem[Kazh]{Kazh} D. Kazhdan, \emph{Connection of the dual space of
    a group with the structure of its closed subgroups},
  Funct. Anal. Appli. {\bf 1} (1967), 64--65.

\bibitem[KL]{KL} D. Kazhdan \& G. Lusztig, \emph{Tensor structures
    arising from affine Lie algebras I}, J. AMS, {\bf 6} (1993),
  905--947; II J. AMS, {\bf 6} (1993), 949--1011; III J. AMS, {\bf 7}
  (1994), 335--381; IV, J. AMS, {\bf 7} (1994), 383--453.



\bibitem[La1]{La1} Y. Laszlo, \emph{Hitchin's and WZW connections are
    the same}, J. Diff. Geom. {\bf 49} (1998), no. 3, 547--576.

\bibitem[Ma]{magnus} W. Magnus, \emph{Rings of fricke characters and
    automorphism groups of free groups}, Math. Zeit., 170:91--103,
  1980.


\bibitem[M1]{M} G. Masbaum, \emph{An element of infinite order in
    TQFT-representations of mapping class groups}. Low-dimensional
  topology (Funchal, 1998), 137--139, Contemp. Math., {\bf 233}, Amer.
  Math. Soc., Providence, RI, 1999.

\bibitem[M2]{M2}G. Masbaum. \emph{Quantum representations of mapping
    class groups}. In: Groupes et G\'{e}om\'{e}trie (Journ\'{e}e
  annuelle 2003 de la SMF). pages 19--36.
  
  \bibitem[MeSe]{MeSe} V.B. Mehta \& C. S. Seshadri, \emph{Moduli of
    Vector Bundles on Curves with Parabolic Structures}, Math. Ann.,
  248: 205--239, (1980).


\bibitem[MS]{MS} G. Moore and N. Seiberg, \emph{Classical and quantum
    conformal field theory}, Comm. Math. Phys. {\bf 123} (1989),
  177--254.

\bibitem[NS1]{NS1} M.S. Narasimhan and C.S. Seshadri,
  \emph{Holomorphic vector bundles on a compact Riemann surface},
  Math. Ann. {\bf 155} (1964) 69--80.


\bibitem[NS2]{NS2} M.S. Narasimhan and C.S. Seshadri, \emph{Stable and
    unitary vector bundles on a compact Riemann surface}, Ann. Math.
  {\bf 82} (1965) 540--67.



\bibitem[R1]{R1} T.R. Ramadas, \emph{Chern-Simons gauge theory and
    projectively flat vector bundles on $M_g$}, Comm. Math. Phys. {\bf
    128} (1990), no. 2, 421--426.


\bibitem[RSW]{RSW} T.R. Ramadas, I.M. Singer and J. Weitsman,
  \emph{Some Comments on Chern -- Simons Gauge Theory},
  Comm. Math. Phys. {\bf 126} (1989) 409-420.


\bibitem[RT1]{RT1} N. Reshetikhin \& V. Turaev, \emph{Ribbon graphs
    and their invariants derived fron quantum groups},
  Comm. Math. Phys.  {\bf 127} (1990), 1--26.

\bibitem[RT2]{RT2} N. Reshetikhin \& V. Turaev, \emph{Invariants of
    $3$-manifolds via link polynomials and quantum groups}, Invent.
  Math. {\bf 103} (1991), 547--597.

\bibitem[Ro]{Ro} J. Roberts, \emph{Irreducibility of some quantum
    representations of mapping class groups}. J. Knot Theory and its
  Ramifications 10 (2001) 763 -- 767.



\bibitem[Sch]{Sch} M. Schlichenmaier, \emph{Berezin-Toeplitz
    quantization and conformal field theory}, Thesis.

\bibitem[Sch1]{Sch1} M. Schlichenmaier, \emph{Deformation quantization
    of compact K\"{a}hler manifolds by Berezin-Toeplitz quantization}.
  In {\em Conf\'{e}rence Mosh\'{e} Flato 1999, Vol. II (Dijon)},
  289--306, Math. Phys. Stud., {\bf 22}, Kluwer Acad. Publ.,
  Dordrecht, (2000), 289--306.


\bibitem[Sch2]{Sch2} M. Schlichenmaier, \emph{Berezin-Toeplitz
    quantization and Berezin transform}.  In {\em Long time behaviour
    of classical and quantum systems (Bologna, 1999)}, Ser. Concr.
  Appl. Math., {\bf 1}, World Sci. Publishing, River Edge, NJ, (2001),
  271--287.

\bibitem[Se]{Segal}G. Segal, \emph{The Definition of Conformal Field
    Theory}, Oxford University Preprint (1992).
    
\bibitem[Th]{thaddeus} M. Thaddeus, \emph{A prefect Morse function on
    the moduli space of flat connections}, Topology, 39(4), 773--787,
  2000.

\bibitem[Th]{Th} W. Thurston, \emph{On the geometry and dynamics of
    diffeomorphisms of surfaces}, Bull of Amer. Math. Soc. 19 (1988),
  417--431.

\bibitem[TUY]{TUY} A. Tsuchiya, K. Ueno \& Y. Yamada, \emph{Conformal
    Field Theory on Universal Family of Stable Curves with Gauge
    Symmetries}, {\em Advanced Studies in Pure Mathmatics}, {\bf 19}
  (1989), 459--566.

\bibitem[T]{T} V. G. Turaev, \emph{Quantum invariants of knots and
    3-manifolds}, de Gruyter Studies in Mathematics, 18.  Walter de
  Gruyter \& Co., Berlin, 1994. x+588 pp. ISBN: 3-11-013704-6

\bibitem[Tuyn]{Tuyn} Tuynman, G.M., \emph{Quantization: Towards a
    comparision between methods}, J. Math. Phys. {\bf 28} (1987),
  2829--2840.

\bibitem[Vi]{Vi} R. Villemoes, \emph{The mapping class group orbit of
    a multicurve}, arXiv:0802.3000v2


\bibitem[Wa]{Walker} K. Walker, \emph{On Witten's 3-manifold
    invariants}, {\em Preliminary version \# 2, Preprint} 1991.


\bibitem[Wi]{W1} E. Witten, \emph{Quantum field theory and the Jones
    polynomial}, Commun. Math. Phys {\bf 121} (1989) 351--98.

\bibitem[Wo]{Woodhouse} N.J. Woodhouse, \emph{Geometric Quantization},
  Oxford University Press, Oxford (1992).

\end{thebibliography}
\end{document}